\documentclass[UTF8]{amsart}
\usepackage{amsthm}
\usepackage[]{amsmath}
\usepackage{amssymb}
\usepackage{enumerate}
\usepackage{tabularx}
\usepackage[]{color}
\usepackage[left=2.6cm, right=2.6cm, top=3.0cm, bottom=3.0cm]{geometry}
\usepackage[colorlinks]{hyperref}
\usepackage{tikz}
\usepackage{multirow}
\usepackage{diagbox}
\usepackage{subcaption}
\usepackage{xcolor}
\usepackage[ruled, vlined]{algorithm2e}

\newtheorem{assumption}{Assumption}
\newtheorem{lemma}{Lemma}
\newtheorem{theorem}{Theorem}

\allowdisplaybreaks[4]

\title[Preconditioned RDA method for EIP]{ Preconditioned
Reconstructed Discontinuous Approximation For Elliptic Interface
Problem on Unfitted Mesh}

\author[R. Li]{Ruo Li} \address{CAPT, LMAM and School of Mathematical
Sciences, Peking University, Beijing 100871, P.R. China}
\email{rli@math.pku.edu.cn}

\author[Q.-C. Liu]{Qicheng Liu} \address{Hangzhou International 
Innovation Institute of Beihang University, Hangzhou 311115, P.R. China}
\email{qcliu@buaa.edu.cn}

\author[F.-Y. Yang]{Fanyi Yang} \address{School of Mathematics
, Sichuan University, Chengdu 610065, P.R. China}
\email{yangfanyi@scu.edu.cn}

\author[S.-H. Zhao]{Shuhai Zhao} \address{Institute of 
Applied Physics and Computational Mathematics, Beijing 100094, P.R. China}
\email{shuhai@pku.org.cn}

\newcommand{\bm}[1]{\boldsymbol{#1}}
\newcommand{\bmr}[1]{\bm{\mr{#1}}}
\newcommand{\lj}{[ \hspace{-2pt} [}
\newcommand{\rj}{] \hspace{-2pt} ]}
\newcommand{\mb}[1]{\mathbb{#1}}
\newcommand{\mc}[1]{\mathcal{#1}}
\newcommand{\mr}[1]{\mathrm{#1}}
\newcommand{\jump}[1]{\lj #1 \rj}
\newcommand{\aver}[1]{ \{#1\}  }
\newcommand{\wt}[1]{ \widetilde{ #1}}
\newcommand{\wh}[1]{ \widehat{ #1}}

\newcommand{\DGenorm}[1]{ \| #1\|_{\mr{DG}}}
\newcommand{\DGsnorm}[1]{ \| #1\|_{*}}
\newcommand{\DGtenorm}[1]{|\!|\!| #1 |\!|\!|_{\mr{DG}}}
\newcommand\anorm[1]{\|#1\|_{\alpha, 2}}
\newcommand\ainnerprod[2]{(#1,#2)_{\alpha, 2}}
\newcommand\eseminorm[1]{|#1|_{\mr{e}}}
\renewcommand{\d}[1]{\mathrm d \boldsymbol{#1}}

\newcommand\NRoman[1]{\uppercase\expandafter{\romannumeral#1}}

\def\Oho{\Omega_{h, 0}}
\def\Ohl{\Omega_{h, 1}}
\def\Ohi{\Omega_{h, i}}
\def\Ohic{\Omega_{h, i}^{\circ}}
\def\OhG{\Omega_h^{\Gamma}}

\def\MTh{\mc{T}_h}
\def\MThi{\mc{T}_{h, i}}
\def\MTho{\mc{T}_{h, 0}}
\def\MThl{\mc{T}_{h, 1}}
\def\MThic{\mc{T}_{h, i}^{\circ}}

\def\MThG{\mc{T}_h^{\Gamma}}

\def\MEh{\mc{E}_h}
\def\MEhI{\mc{E}_h^I}
\def\MEhB{\mc{E}_h^B}
\def\MEhi{\mc{E}_{h, i}}

\def\MEhiI{\mc{E}_{h, i}^{I}}
\def\MEhoI{\mc{E}_{h, 0}^{I}}
\def\MEhlI{\mc{E}_{h, 1}^{I}}

\def\un{\bm{\mathrm{n}}}

\newcommand{\dx}[1]{\mathrm d \boldsymbol{#1}}

\def\dim{\mr{dim}}

\def\mR{\mc{R}}

\def\mP{\mc{P}}
\def\mN{\mc{N}}
\def\mS{\mc{S}}
\def\mD{\mc{D}}
\def\mC{\mc{C}}
\def\mT{\mc{T}}
\def\mA{\mc{A}}
\def\mE{\mc{E}}
\def\MThGic{\mc{T}_{h, i}^{\Gamma, \circ}}

\begin{document}

\maketitle

\begin{abstract}
  In this paper, we develop an efficient preconditioned unfitted
  finite element method for the elliptic interface problem, 
  based on the reconstructed discontinuous approximation.
  The key idea is to impose suitable constraints on the local
  least squares reconstruction. 
  These constraints ensure the stability near cut interface elements
  and, more importantly, establish a norm equivalence between the
  high-order space and the lowest-order piecewise constant space.
  This result allows us to construct an optimal preconditioner
  directly from the lowest-order system on the same unfitted mesh for
  any high-order scheme.
  The resulting method combines a cut discontinuous Galerkin
  formulation with Nitsche's penalty technique.
  The approximation space achieves arbitrarily high order accuracy
  with only one degree of freedom per element.
  We prove optimal error estimates and show that the condition number
  of the preconditioned system is uniformly bounded independently of
  the mesh size, coefficient contrast, and the location of the
  interface relative to the mesh.
  Multigrid algorithms are further designed to efficiently approximate
  the inverse of the lowest-order system matrix.
  Numerical experiments in two and three dimensions confirm the
  optimal convergence rates and demonstrate the robustness and
  efficiency of the proposed preconditioning method.

  \noindent \textbf{keywords}:
  Elliptic Interface Problem; 
  Unfitted Mesh; 
  Cut Finite Element Method;
  Reconstructed Discontinuous Approximation; 
  Multigrid Preconditioning.

\end{abstract}


\section{Introduction}
\label{sec_introduction}
The interface problems serve as fundamental models for a wide range of
physical phenomena involving multiple materials or media.
The governing equations are usually coupled through jump conditions
imposed across the interface. 
As a prototypical example, 
the elliptic interface problems arise in
numerous scientific and engineering applications, including materials
science, porous media flow and multiphase flow 
\cite{Chen1998interface,Leveque1994immersed, Burman2025cut}.
Consequently, the development of numerical methods for elliptic
interface problems has attracted considerable research attention in
recent years.

Finite element methods (FEM) are widely used for solving elliptic
interface problems, and can be roughly classified into 
body-fitted methods and body-unfitted methods. 
In body-fitted methods, the interfaces are required to align with grid
lines, and thus the generation of body-fitted meshes for complex
geometries and interfaces is very challenging and time-consuming. 
We refer to \cite{Chen1998interface, Barrett1987fitted} for some
implementations of body-fitted approaches. 
In contrast, body-unfitted methods have become popular in recent
years, as the mesh generation is completely decoupled from the
geometry description of the interface,
which provides great flexibility in handling complex geometries.
Li proposed the immersed finite element method (IFEM) for the elliptic
interface problem in \cite{Li1998immersed}, which represents a
canonical class of unfitted methods.
The main idea of this method is to handle the discontinuity of the
solution by locally modifying the basis functions or adding some
partial penalty terms on cut elements; we refer to
\cite{Na2014partially, Lin2015partially, Wang2009immersed,
Cao2017immersed, Guo2019higher} for more details.
Another typical unfitted method is the cut finite element method
(CutFEM), also known as the Nitsche-type extended finite element
method (Nitsche-XFEM), proposed by Hansbo and Hansbo in
\cite{Hansbo2002unfittedFEM}.
This method solves the interface problem with two separate finite
element spaces defined on each subdomain, where the jump conditions
are weakly enforced by the Nitsche penalty technique. 
This idea has been applied in many interface problems; see
\cite{Hansbo2014cut,
Guzman2018infsup, Burman2012ficticious, Burman2014fictitious,
Liu2020interface, Li2023curl, Yang2024least} for some applications.
For penalty methods, the small cuts near the interface must be
treated carefully.  
Otherwise, the condition number of the resulting linear system cannot
be uniformly bounded and the convergence may even deteriorate.
In general, certain stabilization strategies are required, such as 
ghost penalty methods, element agglomeration and local
extension techniques; we refer to \cite{Burman2010ghost,
Gurkan2019stabilized, Johansson2013high, Badia2018aggregated,
Neiva2021robust, Huang2017unfitted, Burman2021unfitted} for some
works. 
As in standard FEMs, the discretization usually yields a linear system
with a condition number of $O(h^{-2})$ after stabilization.  Although
this estimate is uniform with respect to the interface position, the
linear system remains ill-conditioned and challenging to solve as the
mesh size tends to zero.
To the best of our knowledge, there are few works on the (multigrid
type) preconditioning method for unfitted methods, especially for
high-order schemes. 
In \cite{Lehrenfeld2017optimal}, a preconditioner for the original
Nitsche-type finite element method \cite{Hansbo2002unfittedFEM} in two
dimensions is proposed, which relies on a stable subspace splitting of
the linear spaces. 
In \cite{Ludescher2020multigrid}, a multigrid preconditioner is
proposed for the stabilized Nitsche-type finite element method, where
the prolongation operator is specially designed for the nonnested
hierarchy of unfitted finite element spaces. 
This method is also established on linear spaces, and a rigorous
analysis of the spectral quality of the preconditioner is missing
\cite{Gross2023analysis}.
The multigrid methods for IFEM can be found in \cite{Chu2024multigrid,
Adams2002immersed}.
We also refer to \cite{Gross2023analysis, Badia2018robust} for 
studies on domain decomposition methods. 

In this paper, we propose a preconditioned discontinuous Galerkin (DG)
method on unfitted meshes for the elliptic interface problem, based
on the reconstructed discontinuous approximation.
In this method, a high-order approximation space is reconstructed by
solving a local least squares fitting problem on every element patch,
with only one degree of freedom per element \cite{Li2016discontinuous,
Li2019eigenvalue, Li2024preconditioned}.
This space is a subspace of the standard DG space and inherits
the DG schemes and their flexibility while significantly reducing the
number of degrees of freedom.
This approximation method was applied to the elliptic interface
problem in \cite{Li2020interface}. 
In the present work, the local least squares fitting problem is
modified by adding suitable constraints, and it is noted that these
constraints are essential in our scheme. 
First, the modification enables us to establish the norm equivalence
property from the physical domain to the active mesh for the
reconstructed space, without additional stabilization mechanisms.
We then prove that the numerical solution achieves optimal convergence
rates under error measurements, and that the resulting discrete linear
system has a condition number of $O(h^{-2})$.
Second, the fact that the reconstructed space maintains the same dimension
regardless of the polynomial degree motivates us to construct a
preconditioner from the lowest-order space, namely the
piecewise constant space. 
Owing to the local constraints, we prove a spectral equivalence between the
lowest-order space and any high-order space, and further show that
the proposed preconditioning method is optimal in the sense that the
condition number of the preconditioned linear system is independent of
the mesh
size, the coefficient contrast and how the interface cuts the
unfitted mesh.  
Low-order preconditioning techniques in standard
finite element methods to precondition the high-order matrix can be
found in \cite{Chalmers2018low, Pazner2023low, Pazner2020efficient}.
Typically, the low-order matrix is constructed on a much finer mesh to
match the dimension of the high-order matrix. 
In contrast, in our method both high-order and lowest-order matrices
are assembled on the same mesh, avoiding the need for any mesh
refinement operator.  
Consequently, the inverse of the matrix arising
from the lowest-order system can serve as an optimal preconditioner
for the high-order scheme. 
For the lowest-order scheme, we follow the idea of the
smoothed aggregation method \cite{Vanek2001convergence} to develop the
multigrid algorithm based on a sequence of successively refined
unfitted meshes. 
The convergence of the multigrid algorithm is also established for the
piecewise constant space.
Numerical tests in two and three dimensions confirm the accuracy of
the proposed scheme and the efficiency of the preconditioning method.
As a numerical observation, our method achieves comparable
numerical error with many fewer degrees of freedom than the cut
discontinuous Galerkin method.
Furthermore, the numerical results demonstrate the robustness of both
the scheme and the preconditioning method for problems with large
coefficient jumps.

The rest of this paper is organized as follows. 
In Section \ref{sec_preliminaries}, the notation related to the mesh
and the interface is given. 
In Section \ref{sec_space}, we introduce the reconstructed
approximation space and present the basic properties of the space. 
In Section \ref{sec_scheme}, we present the penalty schemes for the
elliptic interface problem, along with the corresponding error
estimation in the energy norm and the $L^2$ norm. 
The stability near the interface is also verified in this section.
Section \ref{sec_precondition} is devoted to the preconditioning
method. 
We prove a norm equivalence between the high-order space and the
lowest-order space on the same mesh. 
For the lowest-order linear system, we present two multigrid
algorithms, which can serve as the preconditioner for the high-order
linear system. 
Finally,  in Section \ref{sec_numericalresults} we conduct a series of
numerical experiments in two and three dimensions to validate the
accuracy of the proposed scheme and the efficiency of the
preconditioning method.


\section{Problem setting and preliminaries}
\label{sec_preliminaries}
Let $\Omega \subset \mb{R}^d(d = 2, 3)$ be a convex polygonal
(polyhedral) domain with boundary $\partial \Omega$. 
Let $\Omega_0 \Subset \Omega$ be an open subdomain with 
$C^2$-smooth boundary $\partial \Omega_0$. 
We define $\Gamma := \partial \Omega_0$ as its topological boundary,
which serves as a smooth interface dividing $\Omega$
into two disjoint subdomains. 
Let $\Omega_1 := \Omega \backslash \overline{\Omega}_0$, and we have
that $\Omega_0 \cap \Omega_1 = \varnothing$ and
$\overline{\Omega}_0 \cup \overline{\Omega}_1 = \overline{\Omega}$. 
The problem considered in this paper is the elliptic interface problem,
which seeks the solution $u$ such that
\begin{equation}
  \begin{aligned}
    -\nabla \cdot (\alpha \nabla u) & = f, && \text{in } \Omega_0 \cup
    \Omega_1, \\
    u &= g, && \text{on } \partial \Omega, \\
    \jump{u} = a \un_{\Gamma}, \quad
    \jump{\alpha \nabla u} &= b, && \text{on } \Gamma, 
  \end{aligned}
  \label{eq_problem}
\end{equation}
where $\alpha$ is a piecewise positive constant function in
$\Omega_0 \cup \Omega_1$ that may be discontinuous across the
interface, i.e., $\alpha_i := \alpha|_{\Omega_i} > 0$. 
The jump operators are defined as \eqref{eq_interfaceoperators}.
Given data functions $f \in L^2(\Omega)$, $g \in H^{3/2}(\partial
\Omega)$ and the jump conditions $a \in H^{3/2}(\Gamma)$ and $b \in
H^{1/2}(\Gamma)$, the elliptic interface problem admits a unique
solution $u \in H^2(\Omega_0 \cup \Omega_1)$. We refer to
\cite{Kellogg1972higher, Huang2007uniform} for more details about the
regularity results.

Let $\MTh$ be a quasi-uniform partition of $\Omega$ into a series of
simplices (triangles, tetrahedrons).
For any $K \in \MTh$, we denote by $h_K$ the diameter of $K$
and by $\rho_K$ the radius of the largest ball inscribed in $K$.
Let $h := \max_{K \in \MTh} h_K$ be the mesh size and set
$\rho := \min_{K \in \MTh} \rho_K$. 
The quasi-uniformity of $\MTh$ reads: there exists a
constant $C_\nu$ independent of $h$ such that $h \leq C_\nu \rho$.
For any $K \in \MTh$, we define $\Delta_K := \{ K' \in \MTh: \ 
\overline{K'} \cap \overline{K} \neq \varnothing \}$ as the collection
of the elements sharing a common vertex with $K$, and 
we define $\Delta_{K}^s$ as the set formed by neighbouring elements of
layer $s$ in a recursive manner. 
We initialize $\Delta_{K}^1 := \Delta_K$ and define $\Delta_K^s :=
\bigcup_{K' \in \Delta_{K}^{s-1}} \Delta_{K'}$ for $s \geq 2$.
Since $\MTh$ is quasi-uniform, there exists a constant $C_\Delta$ 
depending on $C_\nu$ such that $K' \subset B(\bm{x}_K,  s C_\Delta 
h_K)(\forall K' \in \Delta_K^s)$, where $B(\bm{z}, r)$ denotes the
disk (ball) centered at $\bm{z}$ with radius $r$, and $\bm{x}_K$
is the barycenter of $K$.
Let $\MEh$ be the set of all $d-1$ dimensional faces of
$\MTh$, and we decompose $\MEh$ into $\MEh = \MEhI \cup \MEhB$, where
$\MEhI$ and $\MEhB$ are collections of all interior faces and 
faces lying on the boundary $\partial \Omega$, respectively.
The faces in $\MEh$ are not required to be aligned with
$\Gamma$. 
For any $e \in \MEh$, we denote by $h_e$ the diameter of $e$.

We further give the notation related to the subdomains $\Omega_i(i =
0,1)$. Let
\begin{displaymath}
  \MThi := \{K \in \MTh \ | \ K \cap \Omega_i \neq \varnothing \},
  \quad \MThic := \{K \in \MThi \ | \ K \subset \Omega_i \}, \quad
  \MThG := \{ K \in \MTh \ | \ K \cap \Gamma \neq \varnothing \},
\end{displaymath}
where $\MThi$ is the minimal set of elements completely covering the
whole domain $\Omega_i$, and $\MThic$ is the set of all interior
elements inside $\Omega_i$, and $\MThG$ is the set of all cut
elements. 
We define $\Ohi :=
\text{Int}(\bigcup_{K \in \MThi} \overline{K} )$, $\Ohic :=
\text{Int}(\bigcup_{K \in \MThic} \overline{K})$, $\OhG :=
\text{Int}(\bigcup_{K \in \MThG} \overline{K})$ as their corresponding
domains.
For the partition $\MThi$, we let $\MEhi$ be the set of all $d - 1$
dimensional faces in $\MThi$, and we let $\MEhiI \subset \MEhi$ 
be the set of all interior faces inside $\Ohi$.
For any $K \in \MTh$ and any $e \in \MEh$, we define $K^i := K \cap
\Omega_i$ and $e^i := e \cap \Omega_i$. For any cut element $K \in
\MThG$, we define $\Gamma_K := K \cap \Gamma$.

We make the following geometrical assumption about the mesh. 
\begin{assumption}
  There exists a constant $s_0$ such that there exist two injective
  relation mappings $(\cdot)^{\sigma_i}(i =0, 1): \MThG \rightarrow
  \MThic$ satisfying $K^{\sigma_i} \in \Delta_K^s \cap \MThic$ with an
  index $s \leq s_0$ for any $K \in \MThG$. 
  \label{as_neighbour}
\end{assumption}
For $i = 0, 1$, since $(\cdot)^{\sigma_i}$ is required to be
injective, $K_-^{\sigma_i} \neq K_+^{\sigma_i}$ for different 
$K_-, K_+ \in \MThG$.
Let $\mc{T}_{h, i}^{\Gamma, 1} := \bigcup_{K \in \MThG} (\Delta_{K}^1
\cap \MThic)$ and  $\mc{T}_{h, i}^{\Gamma, 2} := \bigcup_{K \in
\MThG} (\Delta_{K}^2 \cap \MThic)$.
Because $\Gamma$ is of $C^2$, roughly speaking, there will be
$2\# \MThG \approx 2\# \mc{T}_{h, i}^{\Gamma, 1} \approx \# \mc{T}_{h,
i}^{\Gamma, 2}$ for a sufficiently fine mesh.
Hence, the set $\mc{T}_{h, i}^{\Gamma, 2}$ contains many more
elements than $\MThG$. 
This fact allows us to choose the unique $K^{\sigma_i}$ in
$\Delta_K^2 \cap \MThic$ for $K \in \MThG$, and 
generally there holds $s_0 = 2$ in Assumption \ref{as_neighbour}.
In the computer implementation, for every $K \in \MThG$, if
there exists an element $K' \in  \Delta_{K}^1 \cap \MThic$ that has
not been marked, then we let $K^{\sigma_i} = K'$ and mark $K'$;
otherwise, we choose any element $K' \in  \Delta_{K}^2 \cap \MThic$
that has not been marked as $K^{\sigma_i}$.
We further define the set $\MThGic$ as $\MThGic := \{K' \in \MThic: \
\exists K \in \MThG \text{ such that } K^{\sigma_i} = K' \}$.
For any $K' \in \MThGic$, we formally define an inverse mapping
$(\cdot)^{\varsigma_i}$ by letting $(K')^{\varsigma_i} := K \in \MThG$
such that $K^{\sigma_i} = K'$.

Let us introduce the jump and average operators adopted in the
numerical scheme. 
Let $v$ and $\bm{q}$ be piecewise smooth scalar- and
vector-valued functions on $\MThi(i = 0, 1)$.
For any interior face $e \in \MEhiI$, we let $e = \partial K^+ \cap
\partial K^-$ be shared by two neighbouring elements $K^+, K^- \in
\MThi$, and we let $\un^+, \un^-$ be the unit outward normal vector on
$e$. The jump operators $\jump{\cdot}$ and the average operators
$\aver{\cdot}$ on $e$ are defined as 
\begin{displaymath}
  \begin{aligned}
    \jump{v}|_e & := v^+|_e \un^+ + v^-|_e \un^-, && \jump{\bm{q}}|_e :=
    \bm{q}^+|_e \cdot \un^+ + \bm{q}^-|_e \cdot \un^-,\\
    \aver{v}|_e & := \frac{1}{2} (v^+|_e + v^-|_e ), && \aver{\bm{q}}|_e
    :=\frac{1}{2} (\bm{q}^+|_e + \bm{q}^-|_e), 
  \end{aligned} \quad \forall e \in \MEhiI, 
\end{displaymath}
where $v^{\pm} := v|_{K^{\pm}}$, $\bm{q}^{\pm} := \bm{q}|_{K^{\pm}}$.
For the face $e \in \MEhB$, we let $\un$ be the unit outward normal
vector on $e$, and the trace operators are given as
\begin{displaymath}
  \jump{v}|_e := v|_e \un, \quad \jump{\bm{q}}|_e := \bm{q}|_e \cdot
  \un, \quad \aver{v}|_e := v|_e, \quad \aver{\bm{q}}|_e :=
  {\bm{q}}|_e, \quad \forall e \in \MEhB.
\end{displaymath}
For any cut element $K \in \MThG$, we let $v$ and $\bm{q}$ be the
scalar- and vector-valued functions that may be discontinuous across
$\Gamma_K$.  
In this paper, we use the following harmonic weights as adopted in DG
methods \cite{Cai2011discontinuous, Huang2017unfitted,
Ern2009discontinuous} to obtain the robustness for coefficients with
large jumps: 
\begin{equation}
  w^0 := \frac{\alpha_1}{\alpha_0 + \alpha_1}, \quad w^1 :=
  \frac{\alpha_0}{\alpha_0 + \alpha_1}, \quad
  \aver{\alpha}_{w} := \frac{2\alpha_0 \alpha_1}{\alpha_0 +
  \alpha_1},
  \label{eq_weights}
\end{equation}
where $\alpha_{\min} \leq \aver{\alpha}_w \leq \alpha_{\max}$,
$\aver{\alpha}_w \leq 2 \alpha_{\min}$ with $\alpha_{\min(\max)} :=
\min(\max)(\alpha_0, \alpha_1)$.
The jump and the average operators on $\Gamma$ are defined as 
\begin{equation}
  \begin{aligned}
    \jump{v}|_{\Gamma_K} &:= (v^0|_{\Gamma_K} -
    v^1|_{\Gamma_K})\un_{\Gamma}, && \jump{\bm{q}}|_{\Gamma_K} :=
    (\bm{q}^0|_{\Gamma_K} - \bm{q}^1|_{\Gamma_K}) \cdot \un_{\Gamma},
    \\
    \aver{v}_{w}|_{\Gamma_K} & := w^0 v^0|_{\Gamma_K} + w^1
    v^1|_{\Gamma_K},&&  \aver{\bm{q}}_{w}|_{\Gamma_K} := w^0
    \bm{q}^0|_{\Gamma_K} + w^1 \bm{q}^1|_{\Gamma_K},\\
    \aver{v}^{w}|_{\Gamma_K} & := w^1 v^0|_{\Gamma_K} + w^0
    v^1|_{\Gamma_K},&&  \aver{\bm{q}}^{w}|_{\Gamma_K} := w^1
    \bm{q}^0|_{\Gamma_K} + w^0 \bm{q}^1|_{\Gamma_K},
  \end{aligned} \quad \forall K \in \MThG,
  \label{eq_interfaceoperators}
\end{equation}
where $v^i := v|_{K^i}$, $\bm{q}^i := \bm{q}|_{K^i}$, and 
$\un_{\Gamma}$ denotes the unit normal vector on $\Gamma$
pointing from $\Omega_0$ to $\Omega_1$. 

For a bounded domain $D$, $H^r(D)$ denotes the usual Sobolev spaces
with the regular exponent $r \geq 0$, and the standard definitions of
the inner products, seminorms and norms of $H^r(D)$ are followed.
For $r = 0$, the space $H^0(D)$ coincides with the space $L^2(D)$.
Throughout this paper, $C$ and $C_i(i = 0, 1, \ldots)$ are
generic constants that may vary in different occurrences but are
always independent of $h$, the coefficient $\alpha$ and how the
interface $\Gamma$ intersects the mesh $\MTh$.
In the error estimation, we need the Sobolev extension theory and
the $H^1$ trace estimate on the interface. 
For $i = 0, 1$, we assume that there
exists an extension operator $E^i: H^t(\Omega_i) \rightarrow
H^t(\Omega)(t \geq 2)$ such that \cite[Chapter 5]{Adams2003sobolev}
\begin{equation}
  (E^i v)|_{\Omega_i} = v, \quad \|E^i v \|_{H^s(\Omega)} \leq C \| v
  \|_{H^s(\Omega_i)} , \quad 0 \leq s \leq t, \quad \forall v \in
  H^t(\Omega_i).
  \label{eq_Sobolev}
\end{equation}
The $H^1$ trace estimate on the interface reads: there exists a
constant $C$ such that \cite{Hansbo2002unfittedFEM, Wu2012unfitted}
\begin{equation}
  \| v \|_{L^2(\Gamma_K)}^2 \leq C ( h_K^{-1} \| v \|_{L^2(K)}^2 + h_K
  \| \nabla v \|_{L^2(K)}^2), \quad \forall v \in H^1(K), \quad
  \forall K \in \MThG.
  \label{eq_H1trace}
\end{equation}


\section{Discontinuous Reconstructed Approximation}
\label{sec_space}
In this section, we introduce the approximation space by defining two
linear reconstruction operators $\mR^{m, i}(i = 0, 1)$ in the domain
$\Ohi$.
Let $V_{h, i}:= \{v_h \in L^2(\Ohi): \  v_h|_K \in \mP_0(K), \ \forall
K \in \MThi\}$ be the piecewise constant space over $\MThi$.
The linear operator $\mR^{m, i}$ reconstructs the lowest-order
space $V_{h, i}$ into a piecewise high-order polynomial space over
$\MThi$.
Following the idea in \cite{Li2020interface, Li2024preconditioned},
the construction of $\mR^{m, i}(i = 0, 1)$ includes two steps.

\textbf{Step \NRoman{1}}. We construct a wide element patch $\mS_K^i$
for each element in $\MThi$. The patch consists of some neighbouring
elements around $K$.
We prescribe a threshold $\mN_m$ to control the size of the patch; its
value will be specified later.
Let $\mS_{K, s}^i := \Delta_{K}^s \cap \MThi$, and we first seek a
depth $t$ such that  $\#\mS_{K, t}^i < \mN_m \leq \#\mS_{K, t+1}^i$.
It is noted that $\Delta_K^s$ is defined recursively, and
such $t$ can be readily computed by a recursive algorithm from the
definition of $\mS_{K, s}^i$; we denote by $t_{K, m}^i := t$ the
recursive depth for $K$.
Then, we sort the elements in $\mS_{K, t+1}^i \backslash \mS_{K, t}^i$
by their distances to $K$, where the distance between two elements is
measured by the distance between their barycenters.
We let $\mS_K^i =\mS_{K, t}^i$ and add the first $\mN_m - \#
\mS_{K, t}^i$ elements in $\mS_{K, t+1}^i \backslash \mS_{K, t}^i$ to
$\mS_{K}^i$. 
As a result, we have collected exactly $\mN_m$ elements in $\mS_K^i$,
and there holds $\mS_K^i \subset \MThi$.
We define $\mD_K^i := \text{Int}(\bigcup_{K' \in \mS_K^i}
\overline{K'})$ as the corresponding domain to the patch.
Because $\MThG = \MTho \cap \MThl$,
each cut element $K$ has two patches $\mS^0_K$ and $\mS_K^1$, while
all interior elements have only one patch.

We make the following assumption on patches, which can be easily  
fulfilled by selecting a slightly larger $\mN_m$.
\begin{assumption}
  For any interior element $K \in \MThGic(i = 0, 1)$, there holds
  $K^{\varsigma_i} \in \mS_K^i$.
  \label{as_patch}
\end{assumption}

\textbf{Step \NRoman{2}.} Next, we solve a least squares fitting
problem per element patch in $\MThi$.
For any $K \in \MThi$, we define the set 
\begin{displaymath}
  \mC_K^i := \left\{\begin{aligned}
    &\{\bm{x}_K, \bm{x}_{K^{\varsigma_i}}\}, && \text{for interior
    $K \in \MThGic$}, \\
    & \{\bm{x}_K\}, && \text{for cut $K \in \MThi \backslash
    \MThGic$}. \\
  \end{aligned}\right.
\end{displaymath}
Let $v_{h, i} \in V_{h, i}$ be any piecewise constant function on
$\MThi$, and for any $K \in \MThi$, we solve a constrained least
squares problem on each patch $\mS_K^i$, which depends on barycenters
of all elements in $\mS_K^i$ and the set $\mC_K^i$: 
\begin{equation}
  \begin{aligned}
    \mathop{\arg\min}_{p \in \mP_m(\mD^i_{K})}&  \ \sum_{K' \in
    \mS^i_K} ( p(\bm{x}_{K'}) - v_{h, i}(\bm{x}_{K'}))^2,  \quad
    \text{s.t. } p(\bm{y}) = v_{h, i}(\bm{y}), \quad \forall \bm{y}
    \in \mc{C}_K^i.
  \end{aligned}
  \label{eq_leastsquares}
\end{equation}
We make the following geometrical assumption to ensure the existence
and the uniqueness of the solution to \eqref{eq_leastsquares}.
\begin{assumption}
  For every patch $\mS_K^i(\forall K \in \MThi, i = 0, 1)$, any
  polynomial $p \in \mP_m(\mD_K^i)$ satisfying $p(\bm{x}_{K'}) = 0(
  \forall K' \in \mS_K^i)$ implies $p \equiv 0$.
  \label{as_uni}
\end{assumption}
By Assumption \ref{as_uni}, the prescribed threshold $\mN_m$ is
required to be greater than $\dim(\mP_m)$, and this condition can be
readily satisfied by selecting a bit larger $\mN_m$ in practice.
Moreover, the barycenters of elements in $\mS_K^i$ are required not to
lie on an algebraic curve (surface) of degree $m$.
Let us show that the solution of the problem \eqref{eq_leastsquares}
is unique.
\begin{lemma}
  For each $K \in \MThi(i = 0, 1)$, the problem
  \eqref{eq_leastsquares} admits a unique solution.
  \label{le_unique}
\end{lemma}
\begin{proof}
  Let $p_1$ and $p_2$ be two solutions to \eqref{eq_leastsquares}. 
  Let $q \in \mP_m(\mD_K^i)$ be any polynomial with $q|_{\mC_K^i} =
  0$, and thus $p_i + tq$ meets the constraint of
  \eqref{eq_leastsquares}.
  We have that
  \begin{displaymath}
    \sum_{K' \in \mS_K^i} (p_i(\bm{x}_{K'}) + t q(\bm{x}_{K'}) - v_{h,
    i}(\bm{x}_{K'}))^2 \geq \sum_{K' \in \mS_K^i} (p_i(\bm{x}_{K'}) -
    v_{h, i}(\bm{x}_{K'}))^2.
  \end{displaymath}
  Since $t$ is arbitrary, the above inequality is
  equivalent to 
  \begin{equation}
    \sum_{K' \in \mS_{K}^i} q(\bm{x}_{K'})(p_i(\bm{x}_{K'}) - v_{h,
    i}(\bm{x}_{K'})) = \sum_{K' \in \mS_{K}^i}
    q(\bm{x}_{K'})(p_1(\bm{x}_{K'}) - p_2(\bm{x}_{K'})) = 0.
    \label{eq_qpiuorth}
  \end{equation}
  It is noted that $(p_1 - p_2)|_{\mC_K^i} = 0$.
  In \eqref{eq_qpiuorth}, taking $q = p_1 - p_2$ immediately yields
  that $p_1 - p_2$ vanishes at all barycenters of elements in
  $\mS_K^i$. 
  Then, $p_1 = p_2$ is concluded by Assumption \ref{as_uni}, which
  completes the proof.
\end{proof}
Let $v_{h, \mS_K^i} \in \mP_m(\mD_K^i)$ be the solution of
\eqref{eq_leastsquares}, and it has a linear dependence on the given
function $v_{h, i}$. 
From this fact, we introduce a linear operator $\mR_K^{m, i}:
V_{h, i} \rightarrow \mP_m(\mD_K^i)$ by letting $\mR_K^{m, i} v_{h, i}
:= v_{h, \mS_K^i}$ for any $v_{h, i} \in V_{h, i}$.
Further, 
it is natural to define a global linear operator $\mR^{m, i}$ in a
piecewise manner on $\MThi$, 
\begin{equation}
  \begin{aligned}
    \mR^{m, i}: & V_{h, i} \rightarrow U_{h, i}^m, \\
    & v_{h, i} \rightarrow \mR^{m, i} v_{h, i}, 
  \end{aligned} \quad (\mR^{m, i} v_{h, i})|_K := (\mR_K^{m, i} v_{h,
  i})|_K, \quad \forall K \in \MThi.
  \label{eq_Rmi}
\end{equation}
It can be seen that $\mR^{m, i} v_{h, i}$ is a piecewise
polynomial function of degree $m$. 
Let $U_{h, i}^m := \mR^{m, i} V_{h, i}$ be the image space, which is 
a piecewise polynomial space over $\MThi$. 
Indeed, the reconstructed space $U_{h, i}^m$ will be used as the
approximation space in the scheme to approximate the solution defined
in $\Omega_i$.

We show that the high-order space $U_{h, i}^m$ and the piecewise
constant space $V_{h, i}$ have the same dimension.
\begin{lemma}
  The operator $\mR^{m, i}(i = 0, 1)$ is full-rank and $\dim(U_{h,
  i}^m) = \dim(V_{h, i})$.
  \label{le_fullrank}
\end{lemma}
\begin{proof}
  By the constraint in \eqref{eq_leastsquares}, 
  it follows that $(\mR^{m, i} v_{h, i})(\bm{x}_K) = v_{h, i}(\bm{x}_K)$
  for any $v_{h, i} \in V_{h, i}$, which directly implies that 
  $\mR^{m, i} v_{h, i} =
  0$ for $v_{h, i} \in V_{h, i}$ only if $v_{h, i} = 0$.
  From the linearity of $\mR^{m, i}$, we know that $\mR^{m, i}$ is
  non-degenerate, and thus the dimensions of 
  $V_{h, i}$ and $U_{h, i}^m$ are the same.
  This completes the proof.
\end{proof}
We note that the constraints in the problem
\eqref{eq_leastsquares} are essential in our method, which allow us to
verify the non-degenerate property of the reconstruction operator, to
ensure the stability near the interface and to develop the
preconditioning method based on the lowest-order space for our
numerical scheme.

From Lemma \ref{le_fullrank}, it is straightforward to outline a set 
of basis functions for the reconstructed space $U_{h, i}^m$. 
Let $\{e_{K, i} \}_{K \in \MThi}$ be basis functions of $V_{h, i}$
such that $e_{K, i}(\bm{x}_{K'}) = \delta_{K, K'}$ for any $K' \in
\MThi$, and we let $\lambda_{K, i} := \mR^{m, i} e_{K, i}$.
Since $\mR^{m, i}$ is full-rank, there holds 
$U_{h, i}^m = \text{span}( \{ \lambda_{K, i} \}_{K \in \MThi})$.
For $e_{K, i}$, we have that $e_{K,
i}|_{\mD_{K'}^i} = 0$ for any $K'$ with $K \not\in \mS_{K'}^i$, which
brings $\mR_{K'}^{m, i} e_{K, i} = 0$. 
This fact gives that $\lambda_{K, i}$ is compactly supported by
$\text{supp}(\lambda_{K, i}) = \bigcup_{K'|K \in \mS_{K'}^i}
\overline{K'}$.
By $\{e_{K, i}\}_{K \in \MThi}$, any $v_{h, i} \in V_{h, i}$ has the
expansion $v_{h, i} = \sum_{K \in \MThi} v_{h, i}(\bm{x}_K) e_{K, i}$.
Hence, $\mR^{m, i} v_{h, i}$ has the following expansion
\begin{equation}
  \mR^{m, i} v_{h, i} = \sum_{K \in \MThi} v_{h, i}(\bm{x}_K)
  \lambda_{K, i}, \quad
  \forall v_{h, i} \in V_{h, i}.
  \label{eq_mRmivhi}
\end{equation}
From this expansion, the operator $\mR^{m, i}$ can be directly
extended to the continuous space $C^0(\Ohi)$.
For any $v_{h, i} \in C^0(\Ohi)$, we define $\mR^{m, i} v_{h, i}$ as
\eqref{eq_mRmivhi}, or equivalently, 
$\mR^{m, i} v_{h, i} = \mR^{m, i} \tilde{v}_{h, i}$, where
$\tilde{v}_{h, i} \in V_{h, i}$ is determined by $\tilde{v}_{h,
i}(\bm{x}_K) = v_{h, i}(\bm{x}_K)$ for any $K \in \MThi$.

We next focus on the approximation property of the reconstruction
operator, which plays an important role in the error estimation. 
The following constants are defined to measure the stability of
$\mR^{m, i}$, 
\begin{equation}
  \begin{aligned}
    \Lambda_m & := \max_{i = 0, 1} \max_{K \in \MThi} (1 + t_m
    \Lambda_{m, K, i} \sqrt{\#\mS_{K}^i}), \quad t_m = \max_{i = 0,
    1} \max_{K \in \MThi} t_{K, m}^i, \\
    \Lambda_{m, K, i}^2 & := \max_{p \in \mP_m(\mD_K^i)} \frac{\|p
    \|_{L^2(K)}^2}{h_K^d \sum_{K' \in \mS_K^i} (p(\bm{x}_{K'}))^2},
    \quad \forall K \in \MThi.
  \end{aligned}
  \label{eq_Lambda}
\end{equation}
We state the following stability result.
\begin{lemma}
  For $i = 0, 1$, there holds
  \begin{equation}
    \|\mR_K^{m, i} v_{h, i} \|_{L^2(K)} \leq C \Lambda_m h_K^{d/2}
    \max_{K' \in \mS_K^i} |v_{h, i}(\bm{x}_{K'})|, \quad \forall K \in
    \MThi, \quad \forall v_{h, i} \in C^0(\Ohi) \cup V_{h, i}.
    \label{eq_RKmistability}
  \end{equation}
  \label{le_RKmistability}
\end{lemma}
\begin{proof}
  Let $p = \mR_K^{m, i} v_{h,i}$, and the constraint of
  \eqref{eq_leastsquares} indicates that $p(\bm{x}_K) =
  v_{h,i}(\bm{x}_K)$ for $K \in \MThi \backslash \MThGic$. 
  We define $l(\bm{x}) := v_{h, i}(\bm{x}_K)$ as the constant
  function. 
  For $K \in \MThGic$, there holds $p(\bm{x}_K) =
  v_{h,i}(\bm{x}_K)$ and $p(\bm{x}_{K^{\varsigma_i}}) =
  v_{h,i}(\bm{x}_{K^{\varsigma_i}})$.
  Note that $|\bm{x}_{K} - \bm{x}_{K^{\varsigma_i}}| \geq C h_K$,
  we define a linear function $l(\bm{x})$ 
  passing through $(\bm{x}_K, v_{h,i}(\bm{x}_K))$ and
  $(\bm{x}_{K^{\varsigma_i}}, v_{h,i}(\bm{x}_{K^{\varsigma_i}}))$ such
  that $|\nabla l| \leq \frac{|v_{h,i}(\bm{x}_K) -
  v_{h,i}(\bm{x}_{K^{\varsigma_i}})|}{| \bm{x}_K -
  \bm{x}_{K^{\varsigma_i}}|} \leq C h_K^{-1} \max_{K'\in \mS_K^i} |v_{h,
  i}(\bm{x}_{K'})|$.
  Let $q := p - l$, which satisfies the constraint in
  \eqref{eq_leastsquares} for both cases, and we have that
  \begin{equation}
    \| l \|_{L^{\infty}(\mD_K^i)} \leq C t_{K, m}^i h_K |\nabla l| +
    l(\bm{x}_K) \leq C t_{K, m}^i \max_{K' \in \mS_K^i} |
    v_{h, i}(\bm{x}_{K'})|.
    \label{eq_lLinfty}
  \end{equation}
  Combining such $q$ with \eqref{eq_qpiuorth} yields 
  $\sum_{K' \in \mS_K^i} (p(\bm{x}_{K'}) -
  l(\bm{x}_{K'}))(p(\bm{x}_{K'}) - v_{h,i}(\bm{x}_{K'})) = 0$. 
  This orthogonality property further leads to
  \begin{displaymath}
    \sum_{K' \in \mS_K^i} (p(\bm{x}_{K'}) - l(\bm{x}_{K'}))^2 \leq
    \sum_{K' \in \mS_K^i} (l(\bm{x}_{K'}) - v_{h,i}(\bm{x}_{K'}))^2
    \leq C (t_{K, m}^i)^2 \# \mS_K^i \max_{K'  \in \mS_K^i} |
    v_{h,i}(\bm{x}_{K'})|^2,
  \end{displaymath}
  and
  \begin{align*} 
    \| p \|_{L^2(K)}^2 &\leq C(\| p - l \|_{L^2(K)}^2 + \|l
    \|_{L^2(K)}^2) \leq C ( \Lambda_{m, K, i}^2 h_K^d \sum_{K' \in
    \mS_K^i} (l(\bm{x}_{K'}) - v_{h,i}(\bm{x}_{K'}))^2 + h_K^d \| l
    \|_{L^{\infty}(K)}^2) \\
    & \leq C h_K^d (\Lambda_{m, K, i}^2  (t_{K, m}^i)^2 \# \mS_K^i +
    1) \max_{K' \in \mS_K^i} | v_{h, i}(\bm{x}_{K'})|^2 \leq C
    \Lambda_m^2 \max_{K' \in \mS_K^i}      | v_{h, i}(\bm{x}_{K'})|^2,
  \end{align*}
  which gives the estimate \eqref{eq_RKmistability} and completes the
  proof.
\end{proof}
From the stability property \eqref{eq_RKmistability}, it is
straightforward to obtain the approximation estimate \cite[Theorem
3.3]{Li2012efficient}.
\begin{theorem}
  For $i = 0, 1$, there holds
  \begin{equation}
    \| v - \mR_K^{m, i} v \|_{H^t(K)} \leq C \Lambda_m h_K^{s - t}
    \| v \|_{H^s(\mD_K^i)}, \quad \forall K \in \MThi, \quad \forall v
    \in H^s(\Omega),
    \label{eq_RKmiapp}
  \end{equation}
  where $0 \leq t \leq s - 1$ and $2 \leq s \leq m + 1$.
  \label{th_RKmiapp}
\end{theorem}
It can be seen that the local polynomial
$\mR_K^{m, i} v$ has an optimal convergence rate if $\Lambda_m$
admits an upper bound independent of $h$.
We show that such a bound exists if we select  a wide enough element
patch for every element.
In \cite{Li2020interface, Li2012efficient}, we introduce another
constant 
\begin{displaymath}
  \Theta_{m, K, i} := \max_{p \in \mP_m(\mD_K^i)} \frac{\max_{\bm{x}
  \in \mD_K^i} p(\bm{x})}{\max_{K' \in \mS_K^i} p(\bm{x}_{K'})}, \quad
  \forall K \in \MThi, \quad i = 0, 1.
\end{displaymath}
From the inverse inequality, we have $\Lambda_{m, K, i} \leq
\Theta_{m, K, i} \leq C \Lambda_{m, K, i}$ for any $K \in \MThi$.
We prove that there exists a threshold $\mN_m$ depending on
$m$ and $C_{\nu}$, such that under the condition $\# \mS_K^i \geq
\mN_m$ it holds that $\Theta_{m, K, i} \leq 2$.
In this case, the depth $t_m$ still depends only on $m, C_{\nu}$. 
Consequently, $\Lambda_m$ has a uniform upper bound independent of the
mesh size and how the interface cuts the background mesh.
However, this theoretical threshold $\mN_m$ is usually too large and
impractical in the computer implementation. 
From the definition \eqref{eq_Lambda}, the 
constant $\Lambda_{m, K, i}$ equals
the minimum singular value of a local matrix for $K$,
which can be easily obtained.
More details and some numerical tests are provided in Appendix
\ref{sec_app}. 
It can be seen that $\Lambda_m$ is nearly constant if $\mN_m$ is large
enough.
In the computer implementation, the constant $\Lambda_{m, K, i}$ can
serve as an indicator to show whether the patch $\mS_K^i$ is large
enough. 

To close this section, we combine two operators $\mR^{m, 0}$ and
$\mR^{m, 1}$ to introduce a global operator $\mR^m$ and its
corresponding space $U_h^m$ on the whole domain $\Omega$.
Let $\chi_i$ be the characteristic function of $\Omega_i$, and we
let $V_h := V_{h, 0} \cdot \chi_0 + V_{h, 1} \cdot \chi_1$. 
We define the linear operator $\mR^m$ by $\mR^m v_h := \mR^{m, 0}
v_{h, 0} + \mR^{m, 1} v_{h,1}$ for any $v_h \in V_h$ with the
decomposition $v_h = v_{h, 0} \cdot \chi_0 + v_{h, 1} \cdot \chi_1$, 
and let $U_h^m := \mR^m V_h$ be its image space. 
Equivalently, $U_h^m$ can be written as $U_h^m = U_{h,
0}^m \cdot \chi_0 + U_{h, 1}^m \cdot \chi_1$.
Because any $w_h \in U_h^m$ has a unique decomposition $w_h = w_{h,0}
\cdot \chi_0 + w_{h, 1} \cdot \chi_1$ with $w_{h, i} \in U_{h,i}^m$,
we formally introduce two projection operators $(\cdot)^{\pi_i}(i =
0,1)$ such that $w^{\pi_i}_h := w_{h, i} \in U_{h, i}^m$ for any $w_h
\in U_h^m$.
For any $v \in H^s(\Omega_0 \cup \Omega_1)$, by the Sobolev extension
\eqref{eq_Sobolev} $v_i := v|_{\Omega_i}$ can be extended to the whole
domain $\Omega$. 
Naturally, $v$ admits the decomposition $v = v_0 \cdot \chi_0 + v_1
\cdot \chi_1$. 
We further formally extend the projection operator $(\cdot)^{\pi_i}$ 
to $v$ by setting $v^{\pi_i} := v_i \in H^s(\Omega)$.


\section{Numerical Scheme}
\label{sec_scheme}
We present the numerical scheme for the elliptic interface problem
\eqref{eq_problem}, based on Nitsche's penalty method and the
reconstructed space $U_h^m$. 
We define the following discrete variational problem to seek the
numerical solution $u_h \in U_h^m$: 
\begin{equation}
  a_{h}(u_h, v_h) = l_h(v_h), \quad \forall v_h \in U_h^m,
  \label{eq_dvarproblem}
\end{equation}
where the bilinear form $a_h(\cdot, \cdot)$ is defined as
\begin{align*}
  a_h(v_h, &w_h) :=  \sum_{K \in \MTh} \int_{K^0 \cup K^1} \alpha
  \nabla v_h \cdot \nabla w_h \dx{x} - \sum_{e \in \MEh} \int_{e^0
  \cup e^1} (\aver{\alpha \nabla v_h} \cdot \jump{w_h} + \aver{\alpha
  \nabla w_h} \cdot \jump{v_h}) \dx{s} \\
  - & \sum_{K \in \MThG} \int_{\Gamma_K} ( \aver{\alpha \nabla v_h}_w
  \cdot \jump{w_h} + \aver{\alpha \nabla w_h}_w \cdot \jump{v_h})\dx{s}
    + \sum_{K \in \MThG}\int_{{\Gamma}_K}\frac{\mu \aver{\alpha}_w}{h_K} 
    \jump{v_h}  \cdot \jump{w_h} \dx{s}  \\ 
  + & \sum_{e \in \MEhoI}  \int_e \frac{\mu \alpha_0}{h_e}\,
   \jump{v_h^{\pi_0}} \cdot  \jump{w_h^{\pi_0}} \dx{s} 
  + \sum_{e \in \MEhlI \cup \MEhB}  \int_e \frac{\mu \alpha_1}{h_e}\,
  \jump{v_h^{\pi_1}} \cdot \jump{w_h^{\pi_1}} \dx{s},
\end{align*}
for any $v_h, w_h \in U_h := U_h^m + H^2(\Omega_0 \cup \Omega_1)$, 
and $\mu > 0$ denotes the penalty parameter.
The linear form $l_h(\cdot)$ is given as
\begin{displaymath}
  \begin{aligned}
    l_h(v_h) := \sum_{K \in \MTh} & \int_{K^0 \cup K^1} f v_h \d{x}
    - \sum_{e \in \MEhB} \int_e g \aver{\alpha \nabla v_h} \d{s} 
    + \sum_{e \in \MEhB} \int_e \frac{\mu \alpha}{h_e} g v_h \d{s}\\
    & + \sum_{K \in \MThG} \int_{\Gamma_K} b \aver{v_h}^w \d{s}
    - \sum_{K \in \MThG} \int_{\Gamma_K} a \aver{\alpha \nabla v_h}_w
     \d{s}  + \sum_{K \in \MThG} \int_{\Gamma_K} 
     \frac{\mu \aver{\alpha}_w}{h_K} a\un_{\Gamma} \jump{v_h} \d{s},
  \end{aligned}
\end{displaymath}
for any $v_h \in U_h$.

We begin the error estimation by introducing the following energy
norms for any $v_h \in U_h$:
\begin{align*}
  \eseminorm{v_h}^2 & := \sum_{e \in \MEhoI} \frac{\alpha_0}{h_e}  \|
  \jump{v_h^{\pi_0}} \|_{L^2(e)}^2  +  \sum_{e \in \MEhlI \cup \MEhB}
  \frac{\alpha_1}{h_e} \| \jump{v_h^{\pi_1}} \|_{L^2(e)}^2 + \sum_{K
  \in \MThG} \frac{\aver{\alpha}_w}{h_K} \| \jump{v_h}
  \|_{L^2(\Gamma_K)}^2, \\
  \DGenorm{v_h}^2  &:= \sum_{K \in \MTh} \alpha \| \nabla v_h
  \|_{L^2(K^0 \cup K^1)}^2 + \eseminorm{v_h}^2, \\
  \DGtenorm{v_h}^2 & :=  \DGenorm{v_h}^2 + \sum_{e \in \MEh}
  \frac{h_e}{\alpha} \|
  \aver{ \alpha \nabla v_h} \|_{L^2(e^0 \cup e^1)}^2 + \sum_{K \in
  \MThG} \frac{h_K}{\aver{\alpha}_w}
  \| \aver{\alpha \nabla v_h}_w \|_{L^2(\Gamma_K)}^2, \\
  \DGsnorm{v_h}^2 & := \sum_{K \in \MTho} \alpha_0 \|\nabla
  v_h^{\pi_0} \|_{L^2(K)}^2 +  \sum_{K \in \MThl} \alpha_1 \|
  \nabla v_h^{\pi_1} \|_{L^2(K)}^2 +
  \eseminorm{v_h}^2.
\end{align*}
We show that the above norms are equivalent over the approximation
space $U_h^m$, which is crucial for the stability near the interface.
\begin{lemma}
  There exist constants $C$ such that 
  \begin{equation}
    \DGenorm{v_h} \leq \DGtenorm{v_h} \leq C \DGsnorm{v_h} \leq C
    \Lambda_m \DGenorm{v_h}, \quad \forall v_h \in U_h^m.
    \label{eq_energynorm}
  \end{equation}
  \label{le_energynorm}
\end{lemma}
\begin{proof}
  We mainly prove the last estimate $\DGsnorm{v_h} \leq C \Lambda_m
  \DGenorm{v_h}$ while others are straightforward by trace estimates.
  From the definition, it suffices to prove that for both $i = 0, 1$, 
  there holds $\sum_{K \in {\MThG}} \alpha_i \| \nabla v_h^{\pi_i}
  \|_{L^2(K)}^2 \leq C \Lambda_m^2 \DGenorm{v_h}^2$.
  Let $w_{h, i} \in V_{h, i}$ such that $\mR^{m, i} w_{h,i} =
  v_h^{\pi_i}$.
  For any $K \in \MThi$, we let $v_{K}^{\pi_i} := \mR_K^{m, i} w_{h,
  i}$.
  For any $K \in \MThG$, we split the patch $\mS_K^i$ into $\mS_K^i =
  \mS_K^{i, \Gamma} + \mS_K^{i, \circ}$, where $\mS_K^{i, \Gamma}
  \subset \MThG$ and $\mS_K^{i, \circ} \subset \MThic$ consist of cut
  elements and interior elements, respectively.
  We define a set $\tilde{\mS}_K^{i} := \mS_K^{i, \circ} \cup \{
  (K')^{\sigma_i}: \ K' \in \mS_K^{i,\Gamma} \}$ formed by interior
  elements. 
  Since $w_{h, i}$ is piecewise constant on $\tilde{\mS}_K^i$, 
  we let $w_{\min} := \min_{K' \in \tilde{\mS}_K^i} w_{h, i} |_{K'}$
  and  $w_{\max} := \max_{K' \in \tilde{\mS}_K^i} w_{h, i} |_{K'}$. 
  We derive that 
  \begin{align}
    \| \nabla v_h^{\pi_i}\|_{L^2(K)}^2 &= \| \nabla(\mR^{m, i} w_{h, i}
    - w_{\min} )\|_{L^2(K)}^2  \leq C h_K^{-2} \| \mR^{m, i} (w_{h, i}
    - w_{\min}) \|_{L^2(K)}^2  \nonumber \\
    & \leq C \Lambda_{m, K, i}^2 h_K^{d-2} \sum_{K' \in
    \mS_K^i} ((w_{h, i} - w_{\min})(\bm{x}_{K'}))^2
    \label{eq_nablavhpii} \\
    & = C \Lambda_{m,
    K, i}^2 h_K^{d-2}
    \big(\sum_{K' \in \mS_K^{i, \Gamma}} ((w_{h, i} -
    w_{\min})(\bm{x}_{K'}))^2  + \# \mS_{K}^{i, \circ} (w_{\max} -
    w_{\min})^2\big).  \nonumber
  \end{align}
  For any $K' \in \mS_K^{i, \Gamma}$, 
  there holds $(w_{h,i} - w_{\min})(\bm{x}_{K'}) =
  (v_{(K')^{\sigma_i}}^{\pi_i} - w_{\min})(\bm{x}_{K'})$, and 
  \begin{align*}
    \|(v_{(K')^{\sigma_i}}^{\pi_i} -
    w_{\min})&\|_{L^{\infty}(K')} \leq C
    h_K^{-d/2} \| v_{(K')^{\sigma_i}}^{\pi_i} - w_{\min}
    \|_{L^2(K')} \leq C h_K^{-d/2}  \| v_{(K')^{\sigma_i}}^{\pi_i} -
    w_{\min} \|_{L^2( (K')^{\sigma_i})} \\
    & \leq C h_K^{-d/2} (\| v_{(K')^{\sigma_i}}^{\pi_i} - w_{h,
    i}(\bm{x}_{ (K')^{\sigma_i}}) \|_{L^2( (K')^{\sigma_i})} + \|
    w_{h, i}(\bm{x}_{ (K')^{\sigma_i}}) - w_{\min}
    \|_{L^2({(K')^{\sigma_i}})}) \\
    & \leq C h_K^{-d/2} \| v_{(K')^{\sigma_i}}^{\pi_i} -
    v_{(K')^{\sigma_i}}^{\pi_i}(\bm{x}_{ (K')^{\sigma_i}})
    \|_{L^2({(K')^{\sigma_i}})} + C |w_{\max} - w_{\min}|\\
    & \leq C h_K^{1-d/2} \| \nabla v_{(K')^{\sigma_i}}^{\pi_i} \|_{L^2(
    (K')^{\sigma_i})} + C |w_{\max} - w_{\min}|.
  \end{align*}
  Then, we bound the term $|w_{\max} - w_{\min}|$.
  Since $K_{\max}, K_{\min} \in \MThic$, there exist a sequence of
  interior elements $K_1, K_2, \ldots, K_M(M = O(t))$, such that $K_1
  = K_{\max}, K_M = K_{\min}$, $K_j$ and $K_{j+1}$ sharing a common
  face $e_j$. 
  It can be seen that
  \begin{align}
    |w_{\max} - w_{\min}|^2 = |w_{K_1} - w_{K_M}|^2 \leq C t \sum_{j =
    1}^{M-1} |w_{K_j} - w_{K_{j+1}}|^2.
    \label{eq_wmaxdiffwmin}
  \end{align}
  Let $\bm{x}_{e_j}$ be any point in $e_j$. From the constraint of
  \eqref{eq_leastsquares}, we find that
  \begin{align}
    |w_{K_j} - &w_{K_{j+1}}|^2  = | v_{K_j}^{\pi_i}(\bm{x}_{K_j}) -
    v_{K_{j+1}}^{\pi_i}(\bm{x}_{K_{j+1}})|^2 \nonumber \\
    & \leq C (|  v_{K_j}^{\pi_i}(\bm{x}_{K_j}) -
    v_{K_j}^{\pi_i}(\bm{x}_{e_j}) |^2 + |
    v_{K_{j+1}}^{\pi_i}(\bm{x}_{K_{j+1}}) -
    v_{K_{j+1}}^{\pi_i}(\bm{x}_{e_j}) |^2 + |
    v_{K_j}^{\pi_i}(\bm{x}_{e_j}) -  v_{K_{j+1}}^{\pi_i}(\bm{x}_{e_j})
    |^2) \label{eq_wjumpface} \\
    & \leq C (h_{K_j}^{2-d} \| \nabla v_{K_j}^{\pi_i}\|_{L^2(K_j)}^2 +
    h_{K_{j+1}}^{2-d} \| \nabla v_{K_{j+1}}^{\pi_i}\|_{L^2(K_{j+1})}^2 +
    h_{e_{j}}^{2-d} \| \jump{v_h^{\pi_i}} \|_{L^2(e_j)}^2). \nonumber
  \end{align}
  Collecting all above estimates and summing over all cut elements
  leads to $\sum_{K \in \MThG} \alpha_i \| \nabla v_h^{\pi_i}
  \|_{L^2(K)}^2 \leq C \Lambda_m^2 \DGenorm{v_h}^2$, which completes
  the proof.
\end{proof}
From Lemma \ref{le_energynorm}, the bilinear form is bounded and
coercive under the energy norm $\DGtenorm{\cdot}$.
\begin{lemma}
  Let $a_h(\cdot, \cdot)$ be defined with a sufficiently large $\mu$,
  then 
  \begin{align}
    a_h(v_h, w_h) &\leq C \DGtenorm{v_h} \DGtenorm{ w_h}, \quad \forall
    v_h, w_h \in U_h, 
    \label{eq_ahbound} \\
    a_h(w_h, w_h) & \geq C \Lambda_m^{-2} \DGtenorm{w_h}^2, \quad \forall
    w_h \in U_h^m. \label{eq_ahcoer}
  \end{align}
  \label{le_ahbc}
\end{lemma}
\begin{proof}
  The boundedness \eqref{eq_ahbound} is quite standard by the
  Cauchy-Schwarz inequality. 
  The coercivity can be obtained following the canonical procedure in
  the interior penalty discontinuous Galerkin framework, see
  \cite{Li2020interface, Huang2017unfitted}, i.e., there holds
  \begin{displaymath}
    a_h(w_h, w_h) \geq \DGenorm{w_h}^2 + (\mu - C_0 \tau^{-1})
    \eseminorm{w_h}^2 - C_1 \tau \DGtenorm{w_h}^2, \quad \forall \tau
    > 0.
  \end{displaymath}
  Together with Lemma \ref{le_energynorm}, the estimate
  \eqref{eq_ahcoer} is reached by selecting a sufficiently large
  penalty, which completes the proof.
\end{proof}
\begin{lemma}
  Let $u \in H^2(\Omega_0 \cup \Omega_1)$ be the exact solution to
  \eqref{eq_problem}, and
  let $u_h \in U_h^m$ be the numerical solution to
  \eqref{eq_dvarproblem}, then
  \begin{equation}
    a_h(u - u_h, v_h) = 0, \quad \forall v_h \in U_h^m.
    \label{eq_ahorth}
  \end{equation}
  \label{le_ahorth}
\end{lemma}
\begin{proof}
  The orthogonality \eqref{eq_ahorth} follows from the fact that
  $\jump{u^{\pi_i}}|_{e} = 0$ for any $e \in \MEhi$, combined with the
  identity $\jump{ab} = \aver{a}_w \jump{b} + \jump{a} \aver{b}^w$.
\end{proof}
\begin{theorem}
  Let $a_h(\cdot, \cdot)$ be defined with a sufficiently large $\mu$,
  and let $u \in H^t(\Omega_0 \cup \Omega_1)(t \geq 2)$ be the exact
  solution to \eqref{eq_problem}, and let $u_h \in U_h^m$ be the
  numerical solution to \eqref{eq_dvarproblem}, then
  \begin{align}
    \DGenorm{u - u_h} \leq C \Lambda_m^3 h^s \| \alpha^{\frac12} u
    \|_{H^t(\Omega_0 \cup \Omega_1)},
    \label{eq_DGerror} \\
    \| u - u_h \|_{L^2(\Omega)} \leq C \Lambda_m^4
    \alpha_{\min}^{-\frac12} h^{s+1} \|
    \alpha^{\frac12} u
    \|_{H^t(\Omega_0 \cup \Omega_1)}, 
    \label{eq_L2error}
  \end{align}
  where $s = \min(m, t - 1)$.
  \label{th_error}
\end{theorem}
\begin{proof}
  From the approximation property \eqref{eq_RKmiapp} and the trace
  estimate \eqref{eq_H1trace}, we have that
  \begin{displaymath}
    \inf_{v_h \in U_h^m} \DGtenorm{u - v_h} \leq C \Lambda_m h^s \|
    \alpha^{\frac12} u \|_{H^t(\Omega_0 \cup \Omega_1)}.
  \end{displaymath}
  From Lemma \ref{le_ahbc} and Lemma \ref{le_ahorth}, we conclude that
  for any $v_h \in U_h^m$, there holds
  \begin{align*}
    C \Lambda_m^{-2} \DGtenorm{v_h - u_h}^2 \leq a_{h}(v_h - u_h, v_h
    - u_h) = a_{h}(u - v_h, v_h - u_h)  \leq C \DGtenorm{u - v_h}
    \DGtenorm{v_h - u_h}, 
  \end{align*}
  and combining the triangle inequality yields $\DGtenorm{u - u_h} \leq
  C \Lambda_m^2 \inf_{v_h \in U_h^m} \DGtenorm{u - v_h}$, which leads
  to the estimate \eqref{eq_DGerror}.

  The $L^2$ estimate follows from the duality argument. 
  Let $\psi \in H^2(\Omega) \cap H_0^1(\Omega)$ be the solution of the
  problem 
  \begin{align*}
    - \nabla \cdot (\alpha \nabla \psi)  = u - u_h, \quad \text{in }
    \Omega, \quad \psi = 0, \quad \text{on } \partial \Omega,
  \end{align*}
  with $\| \alpha \psi \|_{H^2(\Omega)} \leq C \| u - u_h \|_{L^2(\Omega)}$.
  Let $\psi_I := \mR^m \psi$, and we derive that
  \begin{align*}
    \| u - u_h \|_{L^2(\Omega)}^2 & = a_h(\psi, u - u_h) = a_h(\psi -
    \psi_I, u - u_h) \leq C \DGtenorm{\psi - \psi_I} \DGtenorm{u -
    u_h} \\
    & \leq Ch^{s+1} \Lambda_m^4 \|\alpha^{\frac12} \psi\|_{H^2(\Omega)} \|
    \alpha^{\frac12} u
    \|_{H^t(\Omega_0 \cup \Omega_1)},
  \end{align*}
  which completes the proof.
\end{proof}


\section{Preconditioning}
\label{sec_precondition}
In this section, we present the preconditioning method to the final
linear system $A_m \bmr{x} = \bmr{b}$ arising from the discrete
variational form \eqref{eq_dvarproblem}. 
We first estimate the condition number to the matrix $A_m$, which
requires the relationship between the energy norm and the $L^2$ norm. 
\begin{lemma}
  For $i = 0, 1$, there holds 
  \begin{equation}
    \begin{aligned}
      \| \alpha_0^{\frac12} v_h^{\pi_0} \|_{L^2(\Oho)} + \|
      \alpha_1^{\frac12} v_h^{\pi_1} \|_{L^2(\Ohl)}  &\leq C
      \DGsnorm{v_h} \\
      & \leq  C h^{-1} (  \| \alpha_0^{\frac12} v_h^{\pi_0}
      \|_{L^2(\Oho)} + \|\alpha_1^{\frac12} v_h^{\pi_1} \|_{L^2(\Ohl)}
      ), \quad \forall v_h \in U_h^m.
    \end{aligned}
    \label{eq_L2DGnorm}
  \end{equation}
  \label{le_L2DGnorm}
\end{lemma}
\begin{proof}
  The upper bound in \eqref{eq_L2DGnorm} is straightforward, following
  the inverse estimate and the trace estimate. 
  For the lower bound, we first state that $\| \alpha^{\frac12} v_h
  \|_{L^2(\Omega)} \leq C \DGenorm{v_h}$ (see \cite[Lemma
  7]{Johansson2013high}).
  Then, it remains to bound the $L^2$ norms for cut elements. 
  For any $K \in \MThG$ and $i = 0, 1$, there exists a sequence of
  elements $K_1, K_2, \ldots, K_M$ such that $K_j \in \MThi$, $K_1 =
  K$, $K_M = K^{\sigma_i}$ and $K_j$ shares a common face $e_j$ with
  $K_{j+1}$.
  We derive that 
  \begin{displaymath}
    \| v_h^{\pi_i} \|_{L^2(K_j)}^2 \leq C( h_{e_j}^{-1} \|
    \jump{v_h^{\pi_i}} \|_{L^2(e_j)}^2 + \| \nabla v_h^{\pi_i}
    \|_{L^2(K_j)}^2 +  \| \nabla v_h^{\pi_i}
    \|_{L^2(K_{j+1})}^2 + \| v_h^{\pi_i} \|_{L^2(K_{j+1})}^2 ), \quad
    1 \leq j \leq M - 1.
  \end{displaymath}
  Hence, there holds 
  \begin{displaymath}
    \| \alpha_i^{\frac12} v_h^{\pi_i} \|_{L^2(K)}^2 \leq C( \|
    \alpha_i^{\frac12} v_h^{\pi_i} \|_{L^2(K^{\sigma_i})}^2 + \sum_{j
    = 1}^M \| \alpha_i^{\frac12} \nabla v_h^{\pi_i} \|_{L^2(K_j)}^2 +
    \sum_{j = 1}^{M-1} \alpha_i h_{e_j}^{-1} \| \jump{v_h^{\pi_i}}
    \|_{L^2(e_j)}^2).
  \end{displaymath}
  Summation over all cut elements yields $\|\alpha_i^{\frac12}
  v_h^{\pi_i} \|_{L^2(\OhG)} \leq C \DGsnorm{v_h}$, which completes
  the proof.
\end{proof}
\begin{theorem}
  There exists a constant $C$ such that
  \begin{equation}
    \kappa(A_m) \leq C \Lambda_m^4 \frac{\alpha_{\max}}{\alpha_{\min}}
    h^{-2}. 
    \label{eq_kappaA}
  \end{equation}
  \label{th_kappaA}
\end{theorem}
\begin{proof}
  Let $\bmr{v}_i \in \mb{R}^{n_{e, i}}$ be the vector
  corresponding to any piecewise constant function $v_{h, i} \in V_{h,
  i}$. Let $\bmr{v} := (\bmr{v}_0, \bmr{v}_1) \in V_h$, there
  holds
  \begin{displaymath}
    \frac{\bmr{v}^T A_m \bmr{v}}{ \bmr{v}^T \bmr{v}} =
    \frac{a_{h}(v_h, v_h) }{ \|\alpha_0^{\frac12} v_h^{\pi_0} \|_{L^2(\Oho)}^2 +  \|
    \alpha_1^{\frac12} v_h^{\pi_1} \|_{L^2(\Ohl)}^2  } \frac{ \|
    \alpha_0^{\frac12}
    v_h^{\pi_0} \|_{L^2(\Oho)}^2 +  \| \alpha_1^{\frac12} v_h^{\pi_1}
    \|_{L^2(\Ohl)}^2}{ \bmr{v}^T \bmr{v}}, \quad \forall \bmr{v} \neq
    \bmr{0} \in V_h.
  \end{displaymath}
  The first term in the right hand side can be bounded by applying
  Lemma \ref{le_ahbc} - \ref{le_L2DGnorm}, that is
  \begin{align*}
    \| v_h^{\pi_i} \|_{L^2(\Ohi)}^2 = \sum_{K \in \MThi} \|
    v_h^{\pi_i} \|_{L^2(K)}^2 \leq \sum_{K \in \MThi} \Lambda_{m, K,
    i}^2 h^d \sum_{K' \in \mS^i_K} (v_h^{\pi_i})(\bm{x}_{K'})^2 
    \leq h^d \Lambda_m^2 \bmr{v}^T \bmr{v},
  \end{align*}
  By the inverse estimate, there holds 
  \begin{align*}
    \bmr{v}^T \bmr{v} = \sum_{K \in \MThi} ((v_h^{\pi_i})(\bm{x}_K))^2
    \leq \sum_{K \in \MThi} \| v_h^{\pi_i} \|_{L^{\infty}(K)}^2 \leq C
    \sum_{K \in \MThi} h_K^{-d} \| v_h^{\pi_i} \|_{L^2(K)}^2 \leq C
    h^{-d} \| v_h^{\pi_i} \|_{L^2(\Ohi)}^2.
  \end{align*}
  Collecting the above estimates immediately indicates that 
  $C \alpha_{\min} \Lambda_m^{-2} \leq \frac{\bmr{v}^T A_m
  \bmr{v}}{\bmr{v}^T \bmr{v}} \leq C \alpha_{\max} \Lambda_m^2 h^{-2}$
  for any vector $\bmr{v}$, which yields the estimate
  \eqref{eq_kappaA} and completes the proof.
\end{proof}
From Theorem \ref{th_kappaA}, the resulting matrix has a condition
number of $O(h^{-2})$ independent of the interface intersecting the
mesh.

In the rest of this section, we present a preconditioning method for
the system $A_m \bmr{x} = \bmr{b}$.
It is noticeable that the reconstructed space $U_h^m$ always has the
same dimensions for all $m \geq 0$. 
This property inspires us to precondition the high-order scheme by the
lowest-order scheme, i.e., $a_h(\cdot, \cdot)$ over the space $V_h
\times V_h$. 
We define the bilinear form $a_{h, 0}$ as follows,
\begin{equation}
  \begin{aligned}
    a_{h, 0}(v_h, w_h) &:= \sum_{e \in \MEhoI} \frac{\alpha_0}{h_e} \int_e
    \jump{v_h^{\pi_0}} \cdot \jump{w_h^{\pi_0}} \dx{s} +  \sum_{e \in
    \MEhlI \cup \MEhB } \frac{\alpha_1}{h_e} \int_e \jump{v_h^{\pi_1}} \cdot
    \jump{w_h^{\pi_1}} \dx{s} \\
    +& \sum_{K \in \MThG} \frac{\{\alpha\}_w}{h_K} \int_{\Gamma_K}
    \jump{v_h} \cdot \jump{w_h} \dx{s}, \quad \forall v_h, w_h \in
    V_h.
  \end{aligned}
  \label{eq_ah0} 
\end{equation}
Here, $a_{h, 0}(v_h, w_h) = a_h(v_h, w_h)$ for $\forall v_h, w_h \in
V_h$ with the penalty parameter fixed as $\mu = 1$.
Let $A_{0}$ be the matrix arising from $a_{h, 0}(\cdot, \cdot)$.
It can be observed that $a_{h, 0}(v_h, v_h) = \DGenorm{v_h}^2 =
\DGsnorm{v_h}^2$ for $\forall v_h \in V_h$, which immediately
indicates $A_{0}$ is invertible.
We will show that $A_0^{-1}$ serves as an optimal
preconditioner for $A_m$ with any $m \geq 1$, 
which is based on the
following equivalence property for the reconstruction operator $\mR^m$
under the energy norm.

\begin{lemma}
  There exist constants $C$ such that
  \begin{equation}
    \DGsnorm{v_h} \leq C \DGsnorm{\mc{R}^m v_h} \leq C \Lambda_m
    \DGsnorm{v_h}, \quad \forall v_h \in V_h.
    \label{eq_equivVh}
  \end{equation}
  \label{le_equivVh}
\end{lemma}
\begin{proof}
  We first prove the lower bound in \eqref{eq_equivVh}.
  For any face $e \in \MEhoI$ shared by two elements $K_-$
  and $K_+$, the jump $\jump{v_h^{\pi_0}}|_{e}$ can be bounded as
  \eqref{eq_wjumpface}, that is
  \begin{displaymath}
    \alpha_0 h_e^{-1} \| \jump{ v_h^{\pi_0}} \|_{L^2(e)}^2 \leq C
    \alpha_0 ( \| \nabla
    \mR^{m, 0} v_h^{\pi_0} \|_{L^2(K_-)}^2  +  \| \nabla
    \mR^{m, 0} v_h^{\pi_0} \|_{L^2(K_+)}^2 + h_e^{-1} \| \jump{
    \mR^{m, 0} v_h^{\pi_0}} \|_{L^2(e)}^2). 
  \end{displaymath}
  A similar estimate can be established for faces $e \in \MEhlI \cup
  \MEhB$.
  For any cut element $K \in \MThG$, we derive that 
  \begin{align*}
    \aver{\alpha}_w h_K^{-1} &\| \jump{ v_h} \|_{L^2(\Gamma_K)}^2 
    \leq C \aver{\alpha}_w h_K^{-1} (\|
    \jump{v_h - \mR^m v_h} \|_{L^2(\Gamma_K)}^2 + \|
    \jump{\mR^m v_h} \|_{L^2(\Gamma_K)}^2), \\
    & \leq C h_K^{-1} ( \alpha_0 \|v_h^{\pi_0} - \mR^{m, 0}_K v_h^{\pi_0}
    \|_{L^2(\Gamma_K)}^2 + \alpha_1 \|v_h^{\pi_1} - \mR^{m, 1}_K v_h^{\pi_1}
    \|_{L^2(\Gamma_K)}^2 +  \aver{\alpha}_w \| \jump{\mR^m v_h}
    \|_{L^2(\Gamma_K)}^2).
  \end{align*}
  For $i = 0, 1$, together with the constraint in
  \eqref{eq_leastsquares} that $\mR_K^{m, i} v_h^{\pi_i}(\bm{x}_K) =
  v_h^{\pi_i}(\bm{x}_K)$, there holds
  \begin{align*}
    \alpha_i h_K^{-1} \|v_h^{\pi_i} - \mR^{m, i}_K v_h^{\pi_i}
    \|_{L^2(\Gamma_K)}^2 \leq C \alpha_i  h_K^{d-2} \|  \mR^{m, i}_K
    v_h^{\pi_i} -  \mR^{m, i}_K v_h^{\pi_i}(\bm{x}_K)
    \|_{L^{\infty}(K)}^2 \leq C \alpha_i \|\nabla \mR^{m, i}_K
    v_h^{\pi_i} \|_{L^2(K)}^2.
  \end{align*}
  Collecting all estimates yields the lower bound in
  \eqref{eq_equivVh}.

  We next turn to the upper bound of \eqref{eq_equivVh}.
  For the trace term $\| \jump{\mR^{m, 0} v_h^{\pi_0}} \|_{L^2(e)}$ on
  $e \in \MEhoI$, we let $e$ be shared by elements $K_+$ and $K_-$ and
  derive that
  \begin{align*}
    \alpha_0 h_e^{-1} \| \jump{\mR^{m, 0} v_h^{\pi_0}} \|_{L^2(e)}^2 \leq C
    \alpha_0 h_e^{-1} (\| \jump{v_h^{\pi_0}} \|_{L^2(e)}^2 +  \|
    \jump{\mR^{m,0} v_h^{\pi_0}- v_h^{\pi_0}  } \|_{L^2(e)}^2).
  \end{align*}
  From the constraint in \eqref{eq_leastsquares}, the second term
  can be bounded by
  \begin{align*}
    h_e^{-1} \| \jump{\mR^{m,0} v_h^{\pi_0}- v_h^{\pi_0}  }
    \|_{L^2(e)}^2 & \leq C h_e^{-1} ( \| \mR_{K_-}^{m,0} v_h^{\pi_0}
    - v_h^{\pi_0}(\bm{x}_{K_-}) \|_{L^2(e)}^2 +  \| \mR_{K_+}^{m,0}
    v_h^{\pi_0} - v_h^{\pi_0}(\bm{x}_{K_+}) \|_{L^2(e)}^2) \\
    & \leq C( \| \nabla \mR_{K_-}^{m,0} v_h^{\pi_0} \|_{L^2(K_-)}^2 +
    \| \nabla \mR_{K_+}^{m,0} v_h^{\pi_0} \|_{L^2(K_+)}^2).
  \end{align*}
  Analogously, the trace $\| \jump{\mR^{m, 1} v_h^{\pi_1}} \|_{L^2(e)}$
  on $e \in \MEhlI \cup \MEhB$ can be estimated by the same procedure.
  For cut elements, there holds
  \begin{align*}
    &\aver{\alpha}_w  h_K^{-1} \| \jump{\mR^m v_h}
    \|_{L^2(\Gamma_K)}^2 \leq C h_K^{-1}  ( \aver{\alpha}_w \| \jump{v_h
    - \mR^m v_h} \|_{L^2(\Gamma_K)}^2 +  \aver{\alpha}_w \|
    \jump{v_h} \|_{L^2(\Gamma_K)}^2 ) \\
    & \leq C h_K^{-1} (\alpha_0 \| \mR_K^{m, 0} v_h^{\pi_0} -
    v_h^{\pi_0}(\bm{x}_K) \|_{L^2(\Gamma_K)}^2 + \alpha_1 \| \mR_K^{m,
    1} v_h^{\pi_1} - v_h^{\pi_1}(\bm{x}_K) \|_{L^2(\Gamma_K)}^2 +
       \aver{\alpha}_w  \| \jump{v_h} \|_{L^2(\Gamma_K)}^2).
  \end{align*}
  Combining above estimates leads to
  \begin{displaymath}
    \DGsnorm{\mR^m v_h}^2 \leq C ( \sum_{K \in \MTho}\alpha_0
    \|\nabla \mR^{m,
    0}_K v_h^{\pi_0} \|_{L^2(K)}^2 +  \sum_{K \in \MThl} \alpha_1
      \|\nabla \mR^{m,
      1}_K v_h^{\pi_1} \|_{L^2(K)}^2) + C \DGsnorm{v_h}^2.
  \end{displaymath}
  It remains to bound the $L^2$ norm for the gradient.
  For any $K \in \MThi$, we let $v_{\min} := \min_{K' \in \mS_K^i}
  (v_h^{\pi_i} |_{K'})$ and $v_{\max} := \max_{K' \in \mS_K^i}
  (v_{h}^{\pi_i}|_{K'})$. 
  As \eqref{eq_nablavhpii}, we find that
  \begin{displaymath}
    \| \nabla \mR_K^{m, i} v_h^{\pi_i} \|_{L^2(K)}^2 \leq C
    \Lambda_{m, K, i}^2 \# \mS_K^i (v_{\max} - v_{\min})^2,
  \end{displaymath}
  and by \eqref{eq_wmaxdiffwmin} - \eqref{eq_wjumpface}, there holds
  \begin{displaymath}
    \sum_{K \in \MThi} \alpha_i \| \nabla \mR^{m, i}_K v_h^{\pi_i} \|_{L^2(K)}^2
    \leq C \Lambda_m^2 \sum_{e \in \MEhiI} \alpha_i h_e^{-1} \|
    \jump{v_h^{\pi_i}}\|_{L^2(e)}^2 \leq C \Lambda_m^2
    \DGsnorm{v_h}^2,
  \end{displaymath}
  which leads to the desired estimate \eqref{eq_equivVh} and completes
  the proof.
\end{proof}
A direct consequence of \eqref{eq_equivVh} is the following estimate
to the preconditioned system.
\begin{theorem}
  There exists a constant $C$ such that 
  \begin{equation}
    \kappa(A_0^{-1} A_m) \leq C  \Lambda_m^4.
    \label{eq_kappaA0Am}
  \end{equation}
  \label{th_kappaA0Am}
\end{theorem}
\begin{proof}
  From \eqref{eq_ahbound}, \eqref{eq_ahcoer} and \eqref{eq_equivVh},
  we deduce that 
  \begin{align*}
    a_{h,0}(v_h, v_h)  = \DGsnorm{v_h}^2 \leq \DGsnorm{\mR^m v_h}^2 &\leq
    C \Lambda_m^{2} a_h(\mR^m v_h, \mR^m v_h) \\
    & \leq C \Lambda_m^2 \DGsnorm{\mR^m v_h}^2  \leq
    C \Lambda_m^4 \DGsnorm{v_h}^2 = C \Lambda_m^4  a_{h,0}(v_h, v_h),
    \quad \forall v_h \in V_h,
  \end{align*}
  which brings \eqref{eq_kappaA0Am} and completes the proof.
\end{proof}
By Theorem \ref{th_kappaA0Am}, the resulting linear system can be
solved by using the Krylov iterative method (e.g. CG and GMRES) with
$A_0^{-1}$ as the preconditioner. 
In addition, this preconditioner is robust to the coefficient.
Generally speaking, in the iterative steps, the matrix-vector product
$\bmr{y} = A_0^{-1} \bmr{z}$ is computed by solving the linear system
$A_0 \bmr{y} = \bmr{z}$. 
Although $A_0$ arises from the lowest-order scheme, it still has a
condition number of $O(h^{-2})$ as Theorem \ref{th_kappaA}.
The second part in our preconditioning method is to provide a fast solver
as an approximation to the inverse $A_0^{-1}$.
We follow the multigrid idea of the smoothed aggregation method
\cite{Vanek2001convergence, Vanek1996algebraic}.

Assume that the mesh $\MTh$ is obtained by successively refining a
coarse (unfitted) mesh $\mc{T}_{h_1}$ for several times, i.e., there
exists a series of meshes $\mT_{h_1}, \ldots, \mT_{h_J}$
such that $\mT_{h_J} = \MTh$ and $\mT_{h_{j+1}}$ is obtained by
globally
refining the mesh $\mT_{h_j}$, and thus $h_j = 2^{1-j} h_1$ and $h_J =
h$.
For $\mT_{h_j}$, we let $\mT_{h_j, i}(i = 0, 1)$, $\Omega_{h_j, i}$ be
the same as in Section \ref{sec_preliminaries}, and
we let $V_{h_j, i}$ be the piecewise constant space on $\mT_{h_j, i}$. 
Note that $V_{h_j, i}$ is defined on the domain $\Omega_{h_j, i}$, and 
there holds $\Omega_{h_j, i} \subset \Omega_{h_k, i}$ for
any $j \geq k$, which further implies $V_{h_k, i} \subset V_{h_j, i}$. 
From this fact, we introduce an operator $I_{k, j}^i(j \geq
k): V_{h_k, i} \rightarrow V_{h_j, i}$ such that 
$I_{k, j}^i v_{h_k, i} := v_{h_k, i}|_{\Omega_{h_j, i}}$  for $\forall
v_{h_k, i} \in V_{h_k, i}$.
We define an inner product $(\cdot, \cdot)_{\alpha, 2}$ as
\begin{displaymath}
  (v, w)_{\alpha, 2} := (\alpha_0 v, w)_{L^2(\Omega_0)} + (\alpha_1 v,
  w)_{L^2(\Omega_1)}, \quad \forall v, w \in L^2(\Omega), 
\end{displaymath}
with the induced norm $\anorm{\cdot}$.
Moreover,
we introduce the operator $I_{j, k}^i(j \geq k) :
V_{h_j, i} \rightarrow V_{h_k, i}$ such that $I_{j, k}^i v_{h_j, i}$
is the projection of $v_{h_j, i}$ into $V_{h_k, i}$ under the norm
$\anorm{\cdot}$
for $\forall v_{h_j, i} \in V_{h_j, i}$.
Then, we have that $(I_{k, j}^i)^T = I_{j, k}^i$ under the inner
product $\ainnerprod{\cdot}{\cdot}$.
We define a global space 
$V_{h_j} := V_{h_j, 0} \cdot \chi_0 + V_{h_j, 1} \cdot \chi_1$ by
concatenating spaces $V_{h_j, 0}$ and $V_{h_j, 1}$, and define the
operator $I_{j, k} := I_{j, k}^0 \cdot \chi_0 + I_{j, k}^1 \cdot
\chi_1$.

To present the multigrid method, we introduce a linear operator 
$\mA: V_{h_J} \rightarrow V_{h_J}$ on the finest level $J$ by
\begin{displaymath}
  \ainnerprod{\mA v_h}{w_h} = a_{h, 0}(v_h, w_h), \quad \forall v_h, w_h
  \in V_{h_J}, 
\end{displaymath}
with the induced norm $ \| v_h \|_{\mA}^2 = \ainnerprod{\mA v_h}{v_h}$.
On each level $j \leq J$, we define the operator $\mA_j$ and the
symmetric prolongator smoother $S_j$ in a recursive manner, which read
\begin{equation}
  \mA_{j} := (S_{j+1} I_{j, j+1})^T  \mA_{j+1} S_{j+1} I_{j, j+1},
  \quad S_j := I - (\lambda_j)^{-1} \mA_j, \quad \mA_J := \mA,
  \label{eq_Al}
\end{equation}
where $\lambda_j$ is a parameter to be specified later.
For any linear operator, we let $\sigma(\cdot)$ and $\rho(\cdot)$ 
be its spectrum and its largest eigenvalue, respectively. 
We will show that $\lambda_j$ can be chosen as 
\begin{equation}
  \lambda_j = 4^{j - J} \lambda, \quad \lambda > \rho(\mA), 
  \label{eq_lambdal}
\end{equation}
and serve as a bound of $\rho(\mA_j)$.
From Theorem \ref{th_kappaA}, the parameter $\lambda$ can be selected
as $\lambda = O(h^{-2})$.
By \eqref{eq_Al}, we present a $\mc{W}$-cycle multigrid algorithm in
Algorithm~\ref{alg_mgsolver}. 
From \eqref{eq_Al}, the operator $\mA_j$ can be rewritten into the
form $\mA_j = Q_j^T \mA_J Q_j$, where $Q_j: V_{h_j}
\rightarrow V_{h_J}$ is defined by
\begin{displaymath}
  Q_j = S_{J}I_{J-1, J} S_{J-1} I_{J-2, J-1} \ldots
  S_{j+1}I_{j,j+1}, \quad 1 \leq j \leq J - 1, \quad Q_J = I. 
\end{displaymath}
By \cite{Vanek2001convergence}, the convergence analysis to
Algorithm~\ref{alg_mgsolver} is established on the following
results.

The first is the weak approximation property of the space $V_{h_j}$.
\begin{lemma}
  There exists a constant $C$ such that
  \begin{equation}
    \min_{w_{h_j} \in V_{h_j}} \anorm{v_h -  w_{h_j}}^2
    \leq C 4^{J - j} h^2 \| v_h \|_{\mA}^2, \quad \forall v_h \in
    V_{h_J}, \quad 1 \leq j \leq J.
    \label{eq_weakerapp}
  \end{equation}
  \label{le_weakerapp}
\end{lemma}
\begin{proof}
  By \cite[Theorem 2.1]{Karakashian2007convergence}, for $v_h^{\pi_i}
  \in V_{h_J, i}$, there exists a function $v_{h, i} \in
  H^1(\Omega_{h_J, i})$ such that
  \begin{displaymath}
    \sum_{K \in \MThi} \alpha_i h_K^{2q-2} | v_h^{\pi_i} - v_{h,
    i}|_{H^q(K)}^2
    \leq C \sum_{e \in \MEhiI} \alpha_i h_e^{-1} \| \jump{
    v_h^{\pi_i}} \|_{L^2(e)}^2 \leq C \DGsnorm{v_h}^2, \quad \forall q
    = 0, 1.
  \end{displaymath}
  Since $v_h^{\pi_i}$ is piecewise constant on $\MTh$, it follows 
  that $\alpha_i | v_{h, i} |_{H^1(\Omega_i)}^2 \leq \alpha_i | v_{h,
  i}|_{H^1(\Omega_{h_J, i})}^2 \leq C \| v_h \|_{\mA}^2$. 
  From Lemma \ref{le_L2DGnorm}, there holds $ \alpha_i \| v_{h, i}
  \|_{L^2(\Omega_i)}^2 \leq C \| v_h \|_{\mA}^2$. 
  Let $v_i$ be the Sobolev extension of $v_{h, i}|_{\Omega_i}$ to
  $\Omega$, which satisfy that
  $v_i|_{\Omega_i} = v_{h, i}$ and $ \alpha_i \| v_i
  \|_{H^1(\Omega)}^2 \leq C \alpha_i \| v_{h, i} \|_{H^1(\Omega_i)}^2
  \leq C \| v_h \|_{\mA}^2$.
  Let $w_{h_j, i}$ be the $L^2$ interpolant of $v_i$
  into the space 
  $V_{h_j, i}$ such that $\| w_{h_j, i} - v_i\|_{L^2(\Omega_{h_j,
  i})} \leq C h_j \| v_i \|_{H^1(\Omega)}$.
  We derive that
  \begin{align*}
    \anorm{v_h - w_{h_j}}^2 & = \sum_{i = 0, 1} \alpha_i  \| v_h^{\pi_i}
    - w_{h_j, i} \|_{L^2(\Omega_i)}^2   
     \leq  C \sum_{i = 0, 1} \alpha_i (\| v_i - w_{h_j, i}
    \|_{L^2(\Omega_i)}^2 +  \| v_i - v_h^{\pi_i} \|_{L^2(\Omega_i)}^2 )
    \leq C  h_j^2 \| v_h \|_{\mA}^2,
  \end{align*}
  which leads to \eqref{eq_weakerapp} and completes the proof.
\end{proof}
Next, we demonstrate that the selection \eqref{eq_lambdal} is
suitable.
\begin{lemma}
  Let $\lambda_j$ be taken as \eqref{eq_lambdal}, there holds 
  \begin{equation}
    \rho(\mA_{j}) \leq \rho(S_{j+1}^T \mA_{j+1} S_{j+1}) < 4^{j
    - J} \lambda = \lambda_j, \quad 1 \leq j \leq J - 1. 
    \label{eq_Ajupper} 
  \end{equation}
  \label{le_Ajupper}
\end{lemma}
\begin{proof}
  We prove the estimate \eqref{eq_Ajupper} for the finest level $j = J
  - 1$.
  In this case, $S_{j+1}$ has the form $S_J = I - (\lambda_J)^{-1}\mA_J$. 
  By the definition of $\mA_j = \mA_{J - 1}$, we deduce that 
  \begin{align*}
    \rho(\mA_{J - 1}) & = \max_{v_{h_{j}} \in V_{h_{j}}} \frac{
    \ainnerprod{Q_{j} 
    v_{h_j}}{\mA_J (Q_{j}  v_{h_j})}  }{ \ainnerprod{v_{h_j}}{
    v_{h_j}}} = \max_{v_{h_j} \in V_{h_j} } \frac{ \ainnerprod{ I_{
    j, J} v_{h_j}}{
    (S_J^T \mA_J S_J ) (I_{j, J} v_{h_j})}}{  \ainnerprod{I_{j, J}
    v_{h_j}}{ I_{j, J} v_{h_j}} } \\ 
    & \leq  \max_{v_{h_J} \in V_{h_J}} \frac{ \ainnerprod{v_{h_J}}
    {(S_J^T \mA_J S_J) v_{h_J}}}{  \ainnerprod{v_{h_J}}{
    v_{h_J}} } \leq \rho(S_J^T \mA_J S_J). 
  \end{align*}
  Since $\lambda_J = \lambda > \rho(\mA_J)$, we know that
  $\rho(\mA_J / \lambda_J) < 1$. 
  By the direct calculation, we obtain 
  \begin{displaymath}
    S_J^T \mA_J S_J = (I - (\lambda_J)^{-1}\mA_J)^2 \mA_J,
  \end{displaymath}
  and
  \begin{align*}
    \rho(S_J^T \mA_J S_J) = \lambda_J \max_{ t \in
    \sigma(\lambda_J^{-1} \mA_J)} (1 - t)^2t \leq \lambda_J
    \max_{t \in (0,1)} (1 - t)^2 t  < 4^{-1} \lambda,
  \end{align*}
  which gives the estimate \eqref{eq_Ajupper} for $j = J - 1$ and
  indicates $\rho(\mA_{J - 1}) \leq \lambda_{J - 1}$. 
  The last inequality follows from the fact that the maximum of
  the function $f(t) = (1 - t)^2t$ on $(0, 1)$ is reached
  at $t_0 = 3^{-1}$, and $f(t) \leq f(t_0) < 4^{-1}$ for any $t \in (0,
  1)$.
  By the same procedure, we conclude the estimate
  \eqref{eq_Ajupper} at level $j - 1$ from the result at $j$, which
  completes the proof.
\end{proof}

\begin{algorithm}[t]
  \caption{$\mc{W}$-cycle Multigrid Solver, MGSolverI($\bmr{y}_j$,
  $\bmr{z}_j$, $j$)}
  \label{alg_mgsolver}
  \KwIn{the initial guess $\bmr{y}_j$, the right hand side
  $\bmr{z}_j$, the level $j$;}
  \KwOut{the solution $\bmr{y}_j$;}
  \If{$j = 1$}{
  $\mA_1 \bmr{y}_1 = \bmr{z}_1$ is solved by the direct method. 

  return $\bmr{y}_1$;
  }

  \If{$j > 1$}{
  pre-smoothing: apply Gauss-Seidel sweep on $\mA_j \bmr{y}_j =
  \bmr{z}_j$;  

  correction on coarse mesh: set $\bmr{\xi} = (S_j
  I_{j-1, j})^T(\bmr{z}_j - \mA_j \bmr{y}_j)$;

  let $\bmr{w}_1 = \bmr{0}$, and update $\bmr{w}_1$ by
  MGSolverI($\bmr{w}_1$, $\bmr{\xi}$, $j - 1$);

  set $\bmr{w}_2 = \text{MGSolverI($\bmr{w}_1$, $\bmr{\xi}$, $j - 1$)}$; 

  set $\bmr{y}_j = \bmr{y}_j + S_j I_{j - 1,j} \bmr{w}_2$;

  post-smoothing: apply Gauss-Seidel sweep on $\mA_j \bmr{y}_j =
  \bmr{z}_j$;

  return $\bmr{y}_j$;
  }
\end{algorithm}

Combining the estimate \eqref{eq_Ajupper} and the definition of $S_j$, 
we have that 
\begin{equation}
  \rho(S_j) = \max_{t \in \sigma(\lambda_j^{-1} \mA_j)} (1 - t 
  ) < 1, \quad \sigma(S_j) \subset(0, 1), \quad 1 \leq j \leq J.
  \label{eq_vSj}
\end{equation}
Then, both $S_j$ and $\mA_j$ are symmetric and positive definite.
From $\mA_j$, we define $\| v_{h_j} \|_{\mA_j}^2 :=
\ainnerprod{v_{h_j}}{ \mA_j v_{h_j}}$ for $\forall v_{h_j} \in
V_{h_j}$. 
We further deduce that
\begin{align}
  \anorm{v_{h_j} - S_j v_{h_j} } & = \anorm{\lambda_j^{-1} \mA_j
  v_{h_j} } 
  \leq \lambda_j^{-1/2} \| v_{h_j} \|_{\mA_j} \leq
  \frac{1}{\sqrt{\rho(\mA_j)}} \| v_{h_j} \|_{\mA_j}, \quad \forall
  v_{h_j} \in V_{h_j}. \label{eq_ISj}
\end{align}
Let $\wt{Q}_j := Q_jI_{J,j}$ on the space $V_h$, we have the following
results.
\begin{lemma}
  There exist constants $C_0, C_1$ such that 
  \begin{equation}
    \| \wt{Q}_j v_h \|_{\mA} \leq \wt{C}(j) \| v_h \|_{\mA}, \quad
    \forall v_h \in V_h, \quad 1 \leq j \leq J,
    \label{eq_Qlbound}
  \end{equation}
  where $\wt{C}(j) = C_0 + C_1j$, and there exist constants $C_2, C_3$
  such that 
  \begin{equation}
    \anorm{(\wt{Q}_{j} - \wt{Q}_{j+1}) v_h} \leq
    \frac{\wh{C}(j)}{ \sqrt{\rho(\mA_{j+1})}} \| v_h \|_{\mA}, \quad
    \forall v_h \in V_h, \quad
    1 \leq j \leq J - 1,
    \label{eq_QlQl1diff}
  \end{equation}
  where $\wh{C}(j) = C_2 + C_3j$.
  \label{le_MGconv}
\end{lemma}
\begin{proof}
For the first estimate \eqref{eq_Qlbound}, we deduce that
\begin{align}
  \| \wt{Q}_{j} v_h \|_{\mA} & = \| Q_{j} I_{J,j} v_h \|_{\mA} = \|
  Q_{j+1} S_{j+1} I_{j,j+1} I_{J,j} v_h \|_{\mA} = \| S_{j+1}
  I_{j,j+1} 
  I_{J,j} v_h \|_{\mA_{j+1}} \label{eq_wtQjvh} \\
  & \leq \| S_{j+1} I_{J,j+1} v_h\|_{\mA_{j+1}} + \| S_{j+1}
  (I_{J,j+1} - I_{j,j+1} I_{J,j}) v_h \|_{\mA_{j+1}}. \nonumber
\end{align}
By \eqref{eq_vSj}, we have that $  \| S_{j+1} I_{J,j+1}
v_h\|_{\mA_{j+1}} \leq \| I_{J,j+1} v_h \|_{\mA_{j+1}} = \|
\wt{Q}_{j+1} v_h \|_{\mA}$.
From the weak approximation property \eqref{eq_weakerapp} and
\eqref{eq_Ajupper}, it can be seen that
\begin{align}
  \| S_{j+1}(I_{J,j+1} - &I_{j,j+1} I_{J,j}) v_h \|_{\mA_{j+1}}^2
  = \ainnerprod{(I_{J,j+1} - I_{j,j+1} I_{J,j}) v_h }{ S_{j+1}^T \mA_{j+1}
  S_{j+1} (I_{J,j+1} - I_{j,j+1} I_{J,j}) v_h }
  \nonumber \\
  & \leq 4^{j - J} \lambda \anorm{ I_{J,j+1}(v_h - I_{J,j} v_h)
   }^2
  \leq  4^{j - J} \lambda \anorm{v_h - I_{J,j} v_h}^2
  \leq C \| v_h \|_{\mA}^2.\label{eq_Sj1}
\end{align}
Combining \eqref{eq_wtQjvh} - \eqref{eq_Sj1} directly leads to the
desired estimate \eqref{eq_Qlbound}.

For the second estimate \eqref{eq_QlQl1diff}, we have that
\begin{align*}
  \anorm{(\wt{Q}_{j} - \wt{Q}_{j+1}) v_h } & = \anorm{Q_{j+1}
  (S_{j+1} I_{j,j+1} I_{J,j} - I_{J,j+1}) v_h }
  \leq \anorm{(S_{j+1} I_{j,j+1} I_{J,j} - I_{J,j+1}) v_h}\\
  & \leq \anorm{
  S_{j+1} (I_{J,j+1} I_{J,j} - I_{J,j+1}) v_h
  } + \anorm{  (I_{J,j+1}  - S_{j+1} I_{J,j+1} ) v_h
  }.
\end{align*}
The first term can be bounded by the weak approximation property
\eqref{eq_weakerapp} and \eqref{eq_vSj}, i.e. 
\begin{align*}
\anorm{S_{j+1} (I_{j,j+1} I_{J,j} -& I_{J,j+1}) v_h } 
\leq \anorm{ I_{j,j+1} (I_{J,j} v_h -  v_h )} \leq
\anorm{v_h - I_{J,j} v_h } \leq \frac{C}{
\sqrt{\rho(\mA_{j+1})} } \| v_h \|_{\mA}.
\end{align*}
The second term can be estimated by \eqref{eq_ISj}, which reads
\begin{align*}
  \anorm{I_{J,j+1} v_h - S_{j+1} I_{J,j+1} v_h} & \leq
  \frac{C}{\sqrt{\rho(\mA_{j+1})}} \| I_{J,j+1} v_h \|_{\mA_{j+1}} =
  \frac{C}{\sqrt{\rho(\mA_{j+1})}} \| Q_{j+1} I_{J,j+1} v_h
  \|_{\mA}
  \\
  & =  \frac{C}{\sqrt{\rho(\mA_{j+1})}} \| \wt{Q}_{j+1} v_h
  \|_{\mA}
  \leq  \frac{C \wt{C}(j)}{\sqrt{\rho(\mA_{j+1})}} \|  v_h \|_{\mA}.
\end{align*}
The above two estimates lead to the estimate \eqref{eq_QlQl1diff},
which completes the proof.
\end{proof}
From \cite[Theorem 3.5]{Vanek2001convergence}, Lemma \ref{le_MGconv}
gives the convergence rate of Algorithm~\ref{alg_mgsolver}.
\begin{theorem}
  For the linear system $A_{0} \bmr{y} = \bmr{z}$, there holds
  \begin{equation}
    \| \bmr{y} - \mathrm{MG}(\wt{\bmr{y}}, \bmr{z})\|_{\mA} \leq ( 1-
    \frac{1}{C(J)})
    \| \bmr{y} - \wt{\bmr{y}} \|_{\mA}, \quad C(J) = O(J^3).
    \label{eq_MGconv}
  \end{equation}
  \label{th_MGconv}
\end{theorem}
As a result, the linear system $A_m \bmr{x} = \bmr{b}$ of
\eqref{eq_dvarproblem} can be solved by the CG method using
Algorithm~\ref{alg_mgsolver} as the preconditioner.

Finally, we present another $\mc{W}$-cycle multigrid algorithm in
Algorithm~\ref{alg_mgsolverII} for the system $A_{0} \bmr{y} =
\bmr{z}$, which is simpler than Algorithm~\ref{alg_mgsolver} and is
closer to the standard geometrical multigrid.
In Algorithm~\ref{alg_mgsolverII}, the smoother $S_j$ is replaced by
the identical operator. 
On each level $j$, we define a bilinear form 
\begin{displaymath}
  \begin{aligned}
    a_{h_j, 0}(v_{h_j}, w_{h_j}) &:= \sum_{e \in \mE_{h_j, 0}^I}
    \frac{\alpha_0}{h_e} \int_e \jump{v_{h_j}^{\pi_0}} \cdot
    \jump{w_{h_j}^{\pi_0}} \dx{s}
    +  \sum_{e \in \mE_{h_j, 1}^I \cup \mE_{h_j}^B }
    \frac{\alpha_1}{h_e} \int_e
    \jump{v_{h_j}^{\pi_1}} \cdot \jump{w_{h_j}^{\pi_1}} \dx{s} \\
    +& \sum_{K \in \mT_{h_j}^{\Gamma}} \frac{\aver{\alpha}_w}{h_K} \int_{\Gamma_K}
    \jump{v_{h_j}} \cdot \jump{w_{h_j}} \dx{s}, \quad \forall v_{h_j},
    w_{h_j} \in V_{h_j},
  \end{aligned}
\end{displaymath}
which is the interior penalty scheme on $\mT_{h_j}$ over piecewise
constant spaces $V_{h_j} \times V_{h_j}$.
Let $A_{0, j}$ be the matrix of $a_{h_j, 0}(\cdot, \cdot)$, and in
Algorithm~\ref{alg_mgsolverII}, we consider the linear system of the
matrix $A_{0, j}$ on each level. 
Numerically, the iterative solver still performs well when
Algorithm~\ref{alg_mgsolverII} is used as the preconditioner.
The convergence study is left for future work.

\begin{algorithm}[t]
  \caption{$\mc{W}$-cycle Multigrid Solver, MGSolverII($\bmr{y}_j$,
  $\bmr{z}_j$, $j$)}
  \label{alg_mgsolverII}
  \KwIn{the initial guess $\bmr{y}_j$, the right hand side
  $\bmr{z}_j$, the level $j$;}
  \KwOut{the solution $\bmr{y}_j$;}
  \If{$j = 1$}{
  $A_{0, 1} \bmr{y}_1 = \bmr{z}_1$ is solved by the direct method. 

  return $\bmr{y}_1$;
  }

  \If{$j > 1$}{
  pre-smoothing: apply Gauss-Seidel sweep on $A_{0, j} \bmr{y}_j =
  \bmr{z}_j$;  

  correction on coarse mesh: set $\bmr{\xi} = 
  I_{j,j-1}(\bmr{z}_j - A_{0, j} \bmr{y}_j)$;

  let $\bmr{w}_1 = \bmr{0}$, and update $\bmr{w}_1$ by
  MGSolverII($\bmr{w}_1$, $\bmr{\xi}$, $j - 1$);

  set $\bmr{w}_2 = \text{MGSolverII($\bmr{w}_1$, $\bmr{\xi}$, $j - 1$)}$; 

  set $\bmr{y}_j = \bmr{y}_j + I_{j - 1, j} \bmr{w}_2$;

  post-smoothing: apply Gauss-Seidel sweep on $A_{0, j} \bmr{y}_j
  = \bmr{z}_j$;

  return $\bmr{y}_j$;
  }
\end{algorithm}


\section{Numerical Results}
\label{sec_numericalresults}
In this section, we carry out a series of numerical tests to
demonstrate the numerical performance of the proposed method and the
efficiency of the preconditioner. 
For all examples, the data functions ${f}$,
${g}$ and the jump functions $a$ and $b$ are derived from the exact
solution. 
The interfaces in tests are described by level set functions. 
In the numerical scheme, the quadrature rules of the
integration on curved domains as 
\begin{displaymath}
  \int_{K^0} f \dx{x}, \quad \int_{K^1} f \dx{x}, \quad
  \int_{\Gamma_K} f \dx{s}, \quad \forall K \in \MThG,
\end{displaymath}
are required. 
We refer to \cite{Cui2019quadratures} for the method
that generates quadrature points and weights for such integrals, and
the subroutines are freely available online.
In two dimensions, the computational domain $\Omega$ is fixed as the
square domain $\Omega = (-1, 1)^2$, and we adopt a family of triangular
meshes with $h = 1/10, \ldots, 1/80$ for all tests, see
Fig.~\ref{fig_domain}.
In three dimensions, the domain $\Omega$ is taken to be the cube
$\Omega = (0, 1)^3$, and we solve the test on tetrahedral meshes with
$h = 1/8, \ldots, 1/64$, see Fig.~\ref{fig_3dinterface}.

\begin{figure}[htp]
  \centering
  \begin{minipage}[t]{0.21\textwidth}
    \centering
    \begin{tikzpicture}[scale=1.25]
      \centering
      \input{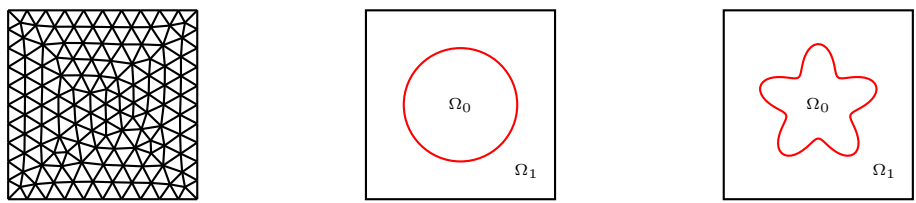}
    \end{tikzpicture}
  \end{minipage}
  \hspace{30pt}
  \begin{minipage}[t]{0.21\textwidth}
    \centering
    \begin{tikzpicture}[scale=1.25]
      \centering
      \node at (0, 0) {\tiny $\Omega_0$};
      \node at (0.7, -0.7) {\tiny $\Omega_1$};
      \draw[thick, black] (-1.0, -1.0) rectangle (1.0, 1.0);
      \draw[thick, red] (0, 0)  circle [radius=0.6];
    \end{tikzpicture}
  \end{minipage}
  \hspace{30pt}
 \begin{minipage}[t]{0.21\textwidth}
    \centering
    \begin{tikzpicture}[scale=1.25]
      \centering
      \node at (0, 0) {\tiny $\Omega_0$};
      \node at (0.7, -0.7) {\tiny $\Omega_1$};
      \draw[thick, black] (-1.0, -1.0) rectangle (1.0, 1.0);
      \draw[thick, domain=0:360, red, samples=120] plot (\x:{(0.5 +
      sin(\x*5)/7)*1.00});
    \end{tikzpicture}
  \end{minipage}
  \caption{The unfitted mesh and the interfaces in two dimensions.}
  \label{fig_domain}
\end{figure}

\begin{figure}[htb]
  \centering
  \includegraphics[width=0.15\textwidth]{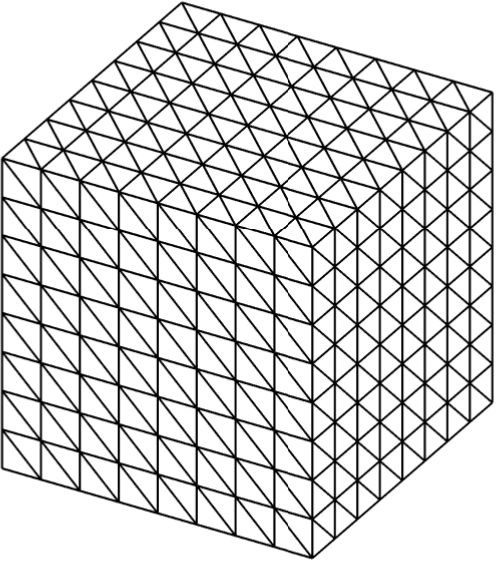}
  \hspace{60pt}
  \includegraphics[width=0.185\textwidth]{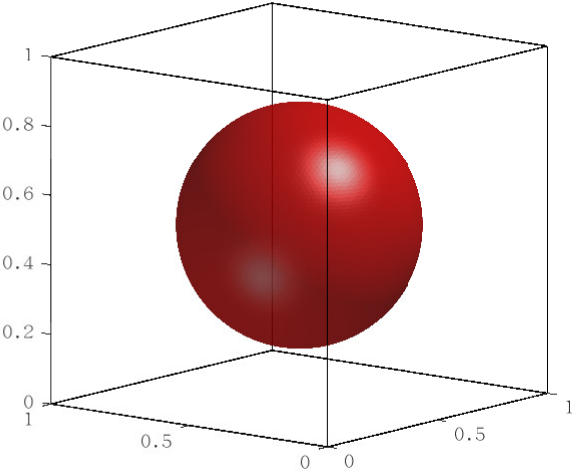}
  \hspace{50pt}
  \includegraphics[width=0.18\textwidth]{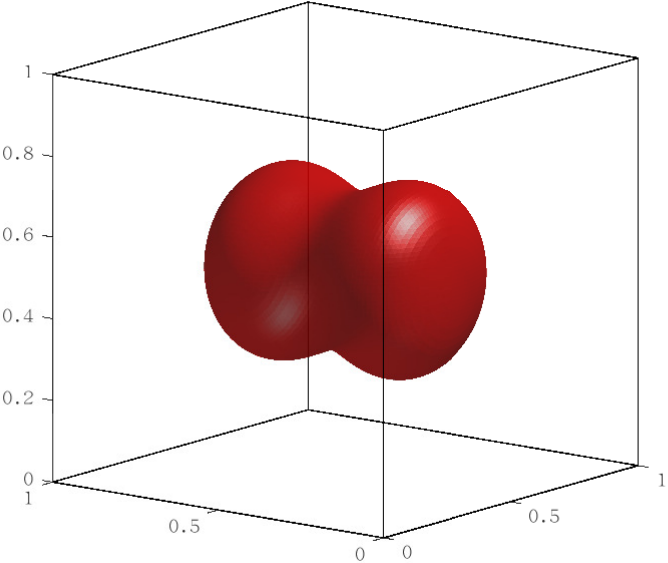}
  \caption{The unfitted mesh and the interfaces in three dimensions.}
  \label{fig_3dinterface}
\end{figure}

\paragraph{\textbf{Example 1.}}
In the first example, we solve an interface problem with a circular
interface centered at the origin with radius $r = 0.6$, see
Fig.~\ref{fig_domain}.
The exact solution and the coefficient are given by
\begin{displaymath}
  \begin{aligned}
    u(x, y) = \begin{cases}
      \sin(\pi x) \sin(\pi y), & \text{in } \Omega_0, \\
      \cos(2\pi x) \cos(4 \pi y), & \text{in } \Omega_1, \\
    \end{cases} \quad \alpha = \begin{cases}
      \alpha_0, & \text{in } \Omega_0, \\
      1, & \text{in } \Omega_1. \\
    \end{cases}
  \end{aligned}
\end{displaymath}
The numerical errors for $\alpha_0 = 10$ are displayed in
Fig.~\ref{fig_ex1err}. 
It can be observed that the errors under both the energy norm and the
$L^2$ norm converge to zero at the optimal rates $O(h^m)$ and
$O(h^{m+1})$, respectively, 
which confirm the
theoretical predictions given in Theorem \ref{th_error}.

\begin{figure}[htbp]
  \centering
  \includegraphics[width=0.32\textwidth]{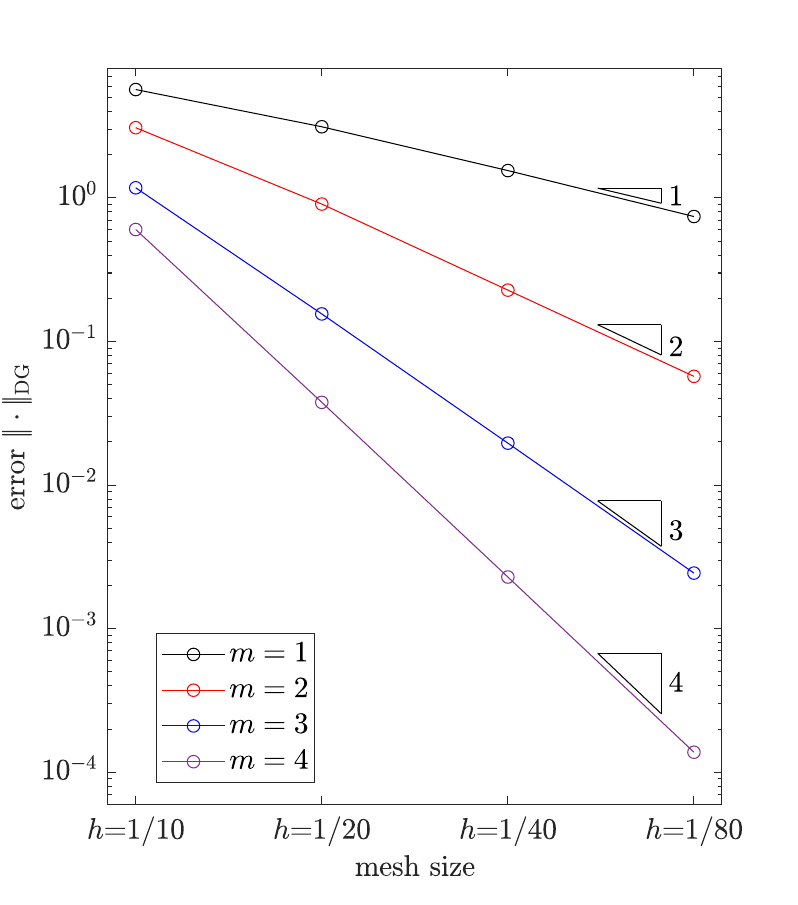}
  \hspace{50pt}
  \includegraphics[width=0.32\textwidth]{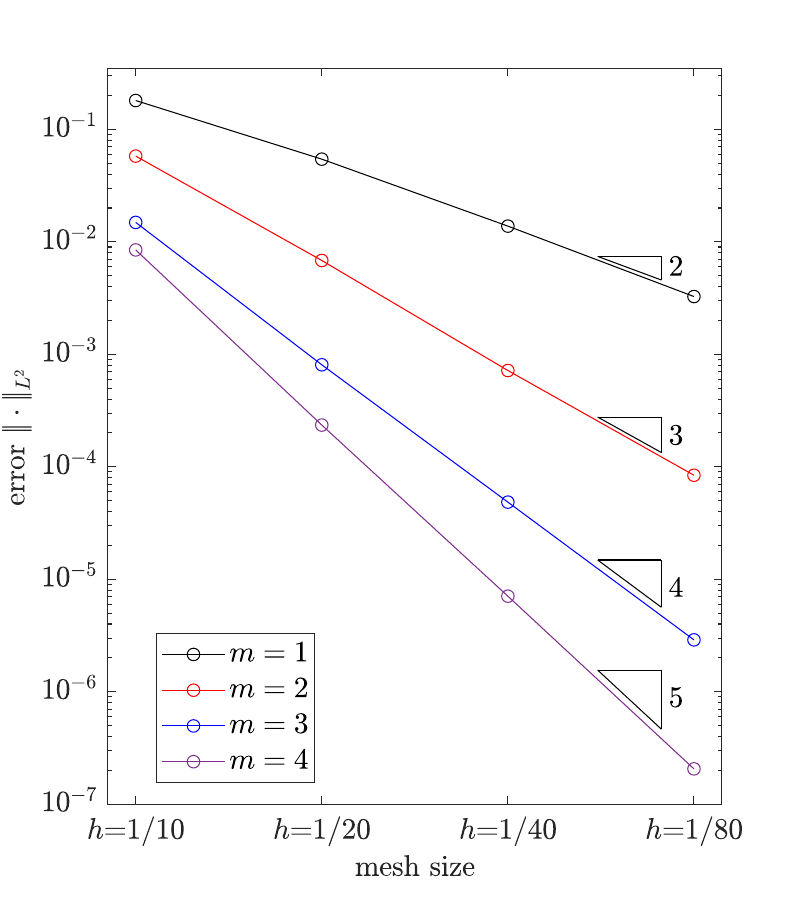}
  \caption{The numerical errors under the energy norm (left)/$L^2$ norm
  (right) in Example 1.}
  \label{fig_ex1err}
\end{figure}

We next focus on the resulting linear system arising from the bilinear
form \eqref{eq_dvarproblem}.
The condition numbers of the linear systems and the preconditioned
systems for all degrees $1 \leq m \leq 4$ are gathered in
Tab.~\ref{tab_cnexample1}.
We find that the condition number of $A_m$ grows like
$O(h^{-2})$ while the condition number of $A_0^{-1} A_m$ increases
only slightly, 
which are consistent with the results in Theorem
\ref{th_kappaA} and Theorem \ref{th_kappaA0Am}.
The iteration counts for standard/preconditioned CG methods are
collected in Tab.~\ref{tab_ex1steps}. 
For both multigrid preconditioning methods, the convergence steps are
observed to remain numerically constant as the unfitted mesh is
refined, demonstrating the effectiveness of our preconditioning
methods.

Furthermore, we show the robustness of the scheme and the
preconditioner for the case involving the coefficient with a large
jump. 
We take $\alpha_0 = 10^0, 10^1, \ldots, 10^4$.
The $L^2$ errors and the iteration counts 
for different coefficients on
the mesh with $h = 1/40$ are summarized in Tab.~\ref{tab_error_alpha}.
Numerical results illustrate that both
the errors and the counts are nearly unchanged 
for all $\alpha_0$, illustrating the robustness of our method.

\begin{table}[htb]
  \centering
  \renewcommand{\arraystretch}{1.28}
  \scalebox{0.78}{
  \begin{tabular}{p{0.5cm}|p{2.0cm}|p{1.8cm}|p{1.8cm}|p{1.8cm}| 
    p{1.8cm}}
    \hline\hline
    \multicolumn{2}{c|}{\diagbox[width=3.2cm, height=0.7cm]{$m$}{$h$}} 
    & 1/10 & 1/20 & 1/40 & 1/80 \\
    \hline
    \multirow{2}{*}{1} & $\kappa(A_{0}^{-1} A_m)$ & 
    7.05 & 7.42 & 8.33 & 8.27 \\ 
    \cline{2-6}
    &$\kappa(A_m)$ & 5.56e+3 & 2.26e+4 & 8.97e+4 & 3.49e+5 \\ 
    \hline
    \multirow{2}{*}{2} & $\kappa(A_{0}^{-1} A_m)$ & 
    27.02 & 30.27 & 31.60 & 38.57 \\ 
    \cline{2-6}
    &$\kappa(A_m)$ & 1.14e+4 & 4.50e+4 & 1.79e+5 & 7.07e+5 \\ 
    \hline
    \multirow{2}{*}{3} & $\kappa(A_{0}^{-1} A_m)$ &
    142.88 & 176.59 & 194.90 & 244.42 \\ 
    \cline{2-6}
    &$\kappa(A_m)$ 
    & 2.45e+4 & 9.27e+4 & 3.92e+5 & 1.56e+6 \\ 
    \hline
    \multirow{2}{*}{4} & $\kappa(A_{0}^{-1} A_m)$ 
    & 303.82 & 392.14 & 408.57 & 458.76 \\ 
    \cline{2-6}
    &$\kappa(A_m)$ & 7.27e+4 & 2.44e+5 & 1.01e+6 & 3.85e+6 \\ 
    \hline \hline
    \end{tabular}
   }
  \caption{The condition numbers of preconditioned/nonpreconditioned
  linear systems in Example 1.}
  \label{tab_cnexample1}
\end{table}

\begin{table}[htb]
  \centering
  \renewcommand{\arraystretch}{1.28}
  \scalebox{0.78}{
  \begin{tabular}{p{0.25cm}|p{3.6cm}|p{1.35cm}|p{1.35cm}|p{1.35cm}|
    p{1.35cm}}
    \hline\hline
    \multirow{2}{*}{$m$} & \diagbox[width=3.9cm, height=0.9cm]{Preconditioner}{$h$}
    & 1/10 & 1/20 & 1/40 & 1/80 \\ \hline
    \multirow{3}{*}{1} 
    & Algorithm \ref{alg_mgsolver} for $A_{0}^{-1}$ 
    & 22 & 22 & 23 & 24 \\ 
    \cline{2-6}
    & Algorithm \ref{alg_mgsolverII}  for $A_{0}^{-1}$ 
    & 22 & 22 & 23 & 23 \\ 
    \cline{2-6}
    & Identity  & 351	& 413	& 758	& 1707 \\ 
    \hline    
    \multirow{3}{*}{2} 
    & Algorithm \ref{alg_mgsolver} for $A_{0}^{-1}$ 
    & 36 & 42	& 47 & 50 \\ 
    \cline{2-6}
    & Algorithm \ref{alg_mgsolverII}  for $A_{0}^{-1}$ 
    & 36 & 41	& 45 & 47 \\ 
    \cline{2-6}
    & Identity & 451	& 856	& 1673	& $>3000$\\ \hline        
    \multirow{3}{*}{3} 
    & Algorithm \ref{alg_mgsolver} for $A_{0}^{-1}$ 
    & 73	& 86	& 88	& 96 \\ 
    \cline{2-6}
    & Algorithm \ref{alg_mgsolverII} for $A_{0}^{-1}$
    & 73	& 86	& 87	& 94	 \\ 
    \cline{2-6}
    & Identity  & 642	& 1199	& 2372	& $>3000$ \\ \hline        
    \multirow{3}{*}{4} 
    & Algorithm \ref{alg_mgsolver} for $A_{0}^{-1}$ 
    & 131	& 138	& 147	& 152 \\ 
    \cline{2-6}
    & Algorithm \ref{alg_mgsolverII} for $A_{0}^{-1}$
    & 131	& 139	& 146	& 149  \\ 
    \cline{2-6}
    & Identity  & 877	& 1661	& $>3000$	& $>3000$ \\ 
    \hline\hline
  \end{tabular}
  }
  \caption{The iteration counts for PCG/CG methods in Example 1.}
  \label{tab_ex1steps}
\end{table}

\begin{table}[htbp]
  \centering
  \renewcommand{\arraystretch}{1.28}
  \scalebox{0.75}{
  \begin{tabular}{c|p{1.3cm}|p{1.3cm}|p{1.3cm}|p{1.3cm}|p{1.3cm}}
    \hline\hline
    \diagbox[width=1.5cm, height=0.7cm]{$m$}{$\alpha_0$} & $10^0$  
    & $10^{1}$ & $10^{2}$ & $10^{3}$ & $10^{4}$\\
    \hline
    1 & 1.32e-2 & 1.30e-2 & 1.22e-2 & 1.18e-2 & 1.18e-2 \\ \hline
    2 & 7.25e-4 & 7.18e-4 & 7.18e-4 & 7.18e-4 & 7.18e-4 \\ \hline 
    3 & 5.26e-5 & 4.85e-5 & 4.85e-5 & 4.86e-5 & 4.86e-5 \\ \hline
    4 & 7.29e-6 & 7.07e-6 & 7.13e-6 & 7.13e-6 & 7.13e-6 \\
    \hline \hline
  \end{tabular}
  \hspace{30pt}
  \begin{tabular}{c|p{1.3cm}|p{1.3cm}|p{1.3cm}|p{1.3cm}|p{1.3cm}}
    \hline\hline
    \diagbox[width=1.5cm, height=0.7cm]{$m$}{$\alpha_0$} & $10^0$  
        & $10^{1}$ & $10^{2}$ & $10^{3}$ & $10^{4}$ \\
    \hline
    1   & 21 & 22 & 22 & 20 & 19 \\ \hline
    2   & 41 & 42 & 43 & 42 & 40 \\ \hline 
    3   & 83 & 86 & 86 & 89 & 87 \\ \hline
    4   & 135& 138& 139& 129& 126\\ \hline \hline
  \end{tabular}
  }
  \caption{The numerical errors(left)/iteration counts of PCG
  solver(right) for different $\alpha_0$ in Example
  1.}
  \label{tab_error_alpha}
\end{table}


\paragraph{\textbf{Example 2.}}
In this example, we test our method by solving an interface problem
with a star-shaped interface consisting of both concave and convex
curve segments, see Fig.~\ref{fig_domain}. 
The interface $\Gamma$ is governed by the level-set function 
\begin{displaymath}
  r = \frac{1}{2} + \frac{\sin 5\theta}{7},
\end{displaymath}
in polar coordinates $(r, \theta)$.
The analytic solution and the coefficient are taken to be the same as
in Example 1.
We depict the convergence histories in Fig.~\ref{fig_ex2err}, 
which again confirm the theoretical results in Theorem \ref{th_error}. 
The condition numbers of $A_m$ and $A_0^{-1} A_m$ are presented in
Tab.~\ref{tab_cnexample2}, and the iteration counts are listed in
Tab.~\ref{tab_ex2steps}. 
As in Example 1, the condition number of the preconditioned system is
nearly constant as the mesh is refined. 
The iteration counts for both preconditioning methods increase
slightly, which agree with the results in 
Theorem \ref{th_kappaA} and Theorem \ref{th_kappaA0Am}.

\begin{figure}[htbp]
  \centering
  \includegraphics[width=0.32\textwidth]{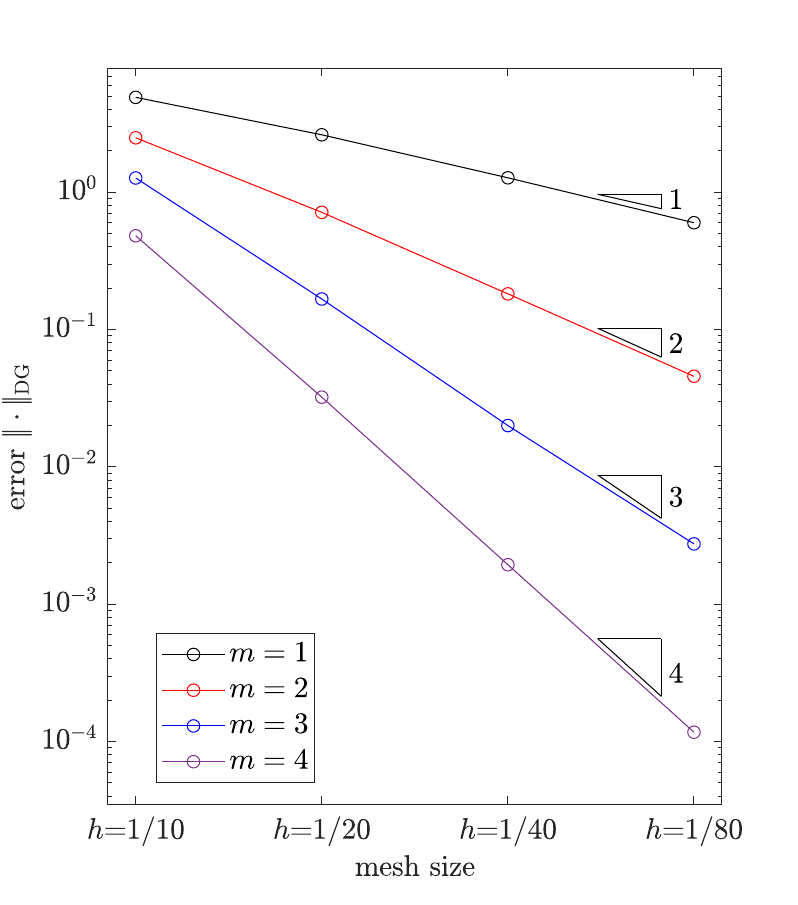}
  \hspace{50pt}
  \includegraphics[width=0.32\textwidth]{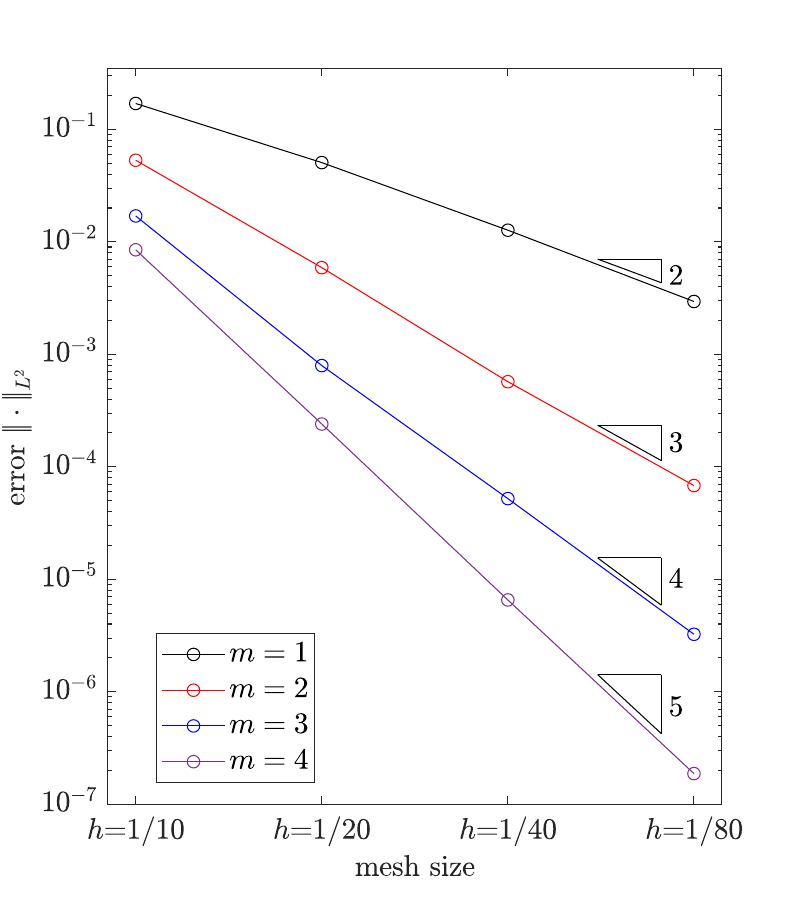}
  \caption{The numerical errors under the energy norm (left)/$L^2$ norm (right) in Example 2.}
  \label{fig_ex2err}
\end{figure}

\begin{table}[htb]
  \centering
  \renewcommand{\arraystretch}{1.28}
  \scalebox{0.78}{
  \begin{tabular}{p{0.5cm}|p{2.0cm}|p{1.8cm}|p{1.8cm}|p{1.8cm}| 
    p{1.8cm}}
    \hline\hline
    \multicolumn{2}{c|}{\diagbox[width=3.2cm, height=0.7cm]{$m$}{$h$}} 
    & 1/10 & 1/20 & 1/40 & 1/80 \\
    \hline
    \multirow{2}{*}{1} & $\kappa(A_{0}^{-1} A_m)$ & 
    6.51 & 6.70 & 7.00 & 7.99 \\ 
    \cline{2-6}
    &$\kappa(A_m)$ & 2.60e+3 & 1.06e+4 & 4.24e+4 & 1.50e+5 \\ 
    \hline
    \multirow{2}{*}{2} & $\kappa(A_{0}^{-1} A_m)$ & 
    16.21 & 30.30 & 31.32 & 38.58 \\ 
    \cline{2-6}
    &$\kappa(A_m)$ & 1.91e+4 & 7.65e+4 & 2.69e+5 & 1.10e+6 \\ 
    \hline
    \multirow{2}{*}{3} & $\kappa(A_{0}^{-1} A_m)$ &
    67.16 & 105.20 & 140.20 & 170.33 \\ 
    \cline{2-6}
    &$\kappa(A_m)$ 
    & 2.76e+4 & 9.28e+4 & 3.55e+5 & 1.60e+6 \\ 
    \hline
    \multirow{2}{*}{4} & $\kappa(A_{0}^{-1} A_m)$ 
    & 225.65 & 260.33 & 388.22 & 439.56\\ 
    \cline{2-6}
    &$\kappa(A_m)$ & 7.35e+4 & 2.70e+5 & 1.01e+6 & 3.81e+6 \\ 
    \hline \hline
    \end{tabular}
   }
  \caption{The condition numbers of preconditioned/nonpreconditioned
  linear systems in Example 2.}
  \label{tab_cnexample2}
\end{table}

\begin{table}[htb]
  \centering
  \renewcommand{\arraystretch}{1.28}
  \scalebox{0.78}{
  \begin{tabular}{p{0.25cm}|p{3.5cm}|p{1.35cm}|p{1.35cm}|p{1.35cm}|
    p{1.35cm}}
    \hline\hline
    \multirow{2}{*}{$m$} & \diagbox[width=3.9cm, height=0.9cm]{Preconditioner}{$h$}
    & 1/10 & 1/20 & 1/40 & 1/80 \\ \hline
    \multirow{3}{*}{1} 
    & Algorithm \ref{alg_mgsolver} for $A_{0}^{-1}$ 
    & 23	& 24	& 23	& 23 \\ 
    \cline{2-6}
    & Algorithm \ref{alg_mgsolverII}  for $A_{0}^{-1}$ 
    & 23	& 23	& 24	& 24 	\\ 
    \cline{2-6}
    & Identity  & 444	& 987	& 1965	& $>3000$ \\ 
    \hline    
    \multirow{3}{*}{2} 
    & Algorithm \ref{alg_mgsolver} for $A_{0}^{-1}$ 
    & 36	& 43	& 46	& 52 \\ 
    \cline{2-6}
    & Algorithm \ref{alg_mgsolverII}  for $A_{0}^{-1}$ 
    & 36	& 43	& 45	& 50 \\ 
    \cline{2-6}
    & Identity  & 509	& 1092	& 2132	& $>3000$ \\ \hline        
    \multirow{3}{*}{3} 
    & Algorithm \ref{alg_mgsolver} for $A_{0}^{-1}$ 
    & 67	& 82	& 90	& 96 \\ 
    \cline{2-6}
    & Algorithm \ref{alg_mgsolverII} for $A_{0}^{-1}$
    & 67	& 82	& 91	& 95 \\ 
    \cline{2-6}
    & Identity  & 627	& 1261	& 2517	& $>3000$ \\ \hline        
    \multirow{3}{*}{4} 
    & Algorithm \ref{alg_mgsolver} for $A_{0}^{-1}$ 
    & 97	& 116	& 141	& 147	 \\ 
    \cline{2-6}
    & Algorithm \ref{alg_mgsolverII} for $A_{0}^{-1}$
    & 97	& 117	& 139	& 144	  \\ 
    \cline{2-6}
    & Identity  & 753	& 1649	& $>3000$ &	$>3000$ \\ 
    \hline\hline
  \end{tabular}
  }
  \caption{The iteration counts for PCG/CG methods in Example 2.}
  \label{tab_ex2steps}
\end{table}

\paragraph{\textbf{Example 3.}}
In this test, we solve a three-dimensional problem and the interface
$\Gamma$ is a sphere centered at $(0.5, 0.5, 0.5)$ with the radius $r
= 0.35$ (see Fig.~\ref{fig_3dinterface}).
The exact solution and the coefficient are selected as
\begin{displaymath}
  u(x, y, z) = \begin{cases}
    \sin(\pi x) \sin(\pi y) \sin(\pi z), & \text{in } \Omega_0, \\
    e^{x + y + z}, & \text{in } \Omega_1, \\
  \end{cases} \quad \alpha = 1 \quad \text{in } \Omega.
\end{displaymath}
The convergence histories are plotted in Fig.~\ref{fig_ex3err}, and
the numerical results demonstrate that the convergence rates under
both error measurements are optimal, which confirms the error
estimation in three dimensions.
The condition numbers and the CG iteration counts for preconditioned
systems and nonpreconditioned systems are collected in
Tab.~\ref{tab_cnexample3} and Tab.~\ref{tab_ex3steps}. 
It can be seen that the condition numbers of the preconditioned linear
system are almost constant, and the CG iteration counts increase very
slightly.
These numerical observations demonstrate the accuracy and the
efficiency of our method in three dimensions.

\begin{figure}[htbp]
  \centering
  \includegraphics[width=0.32\textwidth]{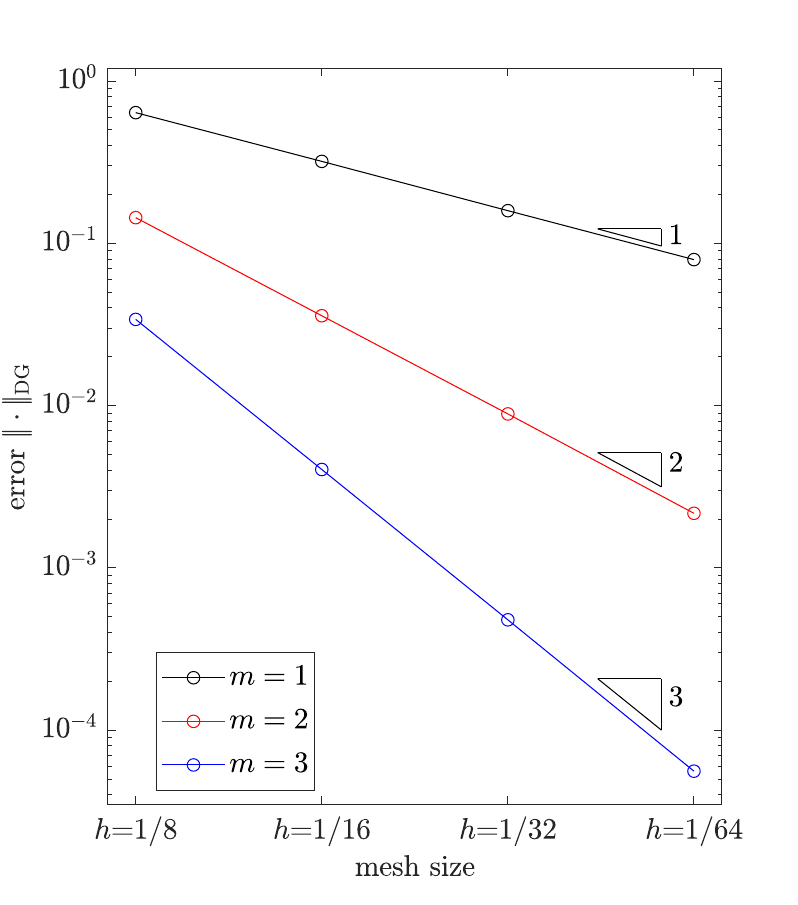}
  \hspace{50pt}
  \includegraphics[width=0.32\textwidth]{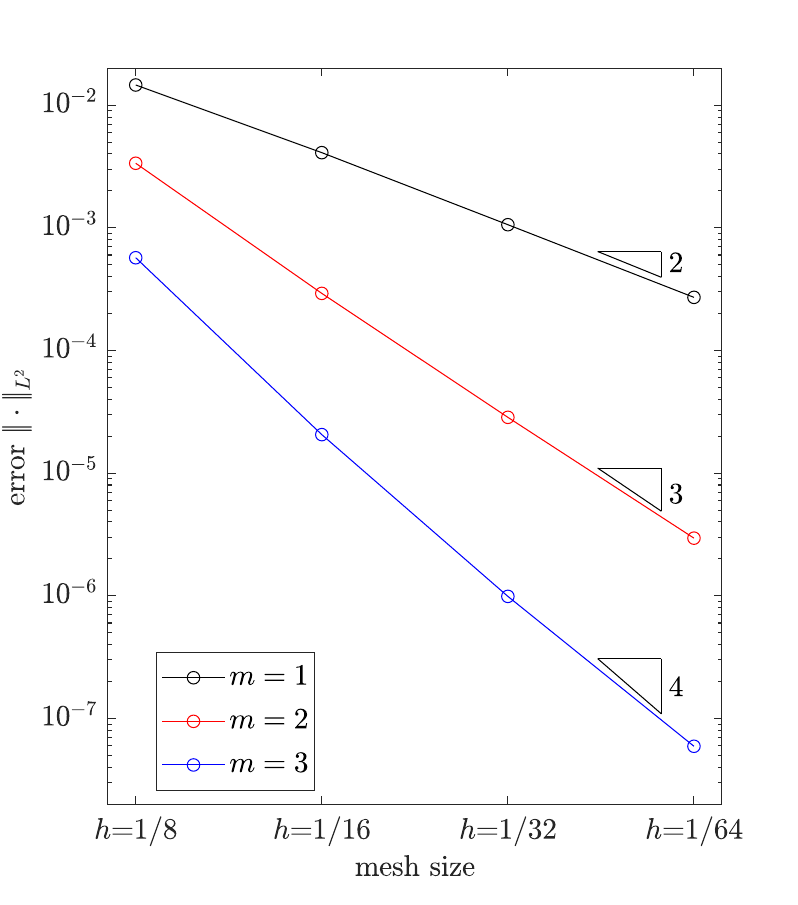}
  \caption{The numerical errors under the energy norm (left)/$L^2$ norm (right) in Example 3.}
  \label{fig_ex3err}
\end{figure}

\begin{table}[htb]
  \centering
  \renewcommand{\arraystretch}{1.28}
  \scalebox{0.78}{
  \begin{tabular}{p{0.5cm}|p{2.0cm}|p{1.8cm}|p{1.8cm}|p{1.8cm}| 
    p{1.8cm}}
        \hline\hline
    \multicolumn{2}{c|}{\diagbox[width=3.2cm, height=0.7cm]{$m$}{$h$}} 
    & 1/8 & 1/16 & 1/32 & 1/64 \\
    \hline
    \multirow{2}{*}{1} & $\kappa(A_{0}^{-1} A_m)$ & 
    21.07 & 31.88 & 37.99 & 41.13 \\ 
    \cline{2-6}
    &$\kappa(A_m)$ & 6.879e+2 & 2.688e+3 & 1.005e+4 & 3.603e+4 \\ 
    \hline
    \multirow{2}{*}{2} & $\kappa(A_{0}^{-1} A_m)$ &
    86.69 & 119.04 & 153.90 & 162.77 \\ 
    \cline{2-6}
    &$\kappa(A_m)$ 
    & 1.472e+3 & 5.556e+3 & 1.944e+4 & 7.399e+4 \\ 
    \hline
    \multirow{2}{*}{3} & $\kappa(A_{0}^{-1} A_m)$ 
    & 421.56 & 665.79 & 707.23 & 778.74 \\ 
    \cline{2-6}
    &$\kappa(A_m)$ & 3.181e+3 & 1.162e+4 & 4.042e+4 & 1.562e+5 \\ 
    \hline \hline
    \end{tabular}
   }
  \caption{The condition numbers of preconditioned/nonpreconditioned
  linear systems in Example 3.}
  \label{tab_cnexample3}
\end{table}

\begin{table}[htb]
  \centering
  \renewcommand{\arraystretch}{1.28}
  \scalebox{0.78}{
  \begin{tabular}{p{0.25cm}|p{3.5cm}|p{1.35cm}|p{1.35cm}|p{1.35cm}|
    p{1.35cm}}
    \hline\hline
    \multirow{2}{*}{$m$} & \diagbox[width=3.9cm, height=0.9cm]{Preconditioner}{$h$}
    & 1/8 & 1/16 & 1/32 & 1/64 \\ \hline
    \multirow{3}{*}{1} 
    & Algorithm \ref{alg_mgsolver} for $A_{0}^{-1}$ 
    & 50	& 58	& 68	& 76 \\ 
    \cline{2-6}
    & Algorithm \ref{alg_mgsolverII}  for $A_{0}^{-1}$ 
    &  50	& 57	& 64	& 70 \\ 
    \cline{2-6}
    & Identity  & 159 & 787 & 1357 & 2674\\ \hline        
    \multirow{3}{*}{2} 
    & Algorithm \ref{alg_mgsolver} for $A_{0}^{-1}$ 
    & 90	& 112	& 125	& 119 \\ 
    \cline{2-6}
    & Algorithm \ref{alg_mgsolverII} for $A_{0}^{-1}$
    & 90	& 110	& 121	& 114 \\ 
    \cline{2-6}
    & Identity 
    & 109 & 1005 & 2052 & $>3000$ \\ \hline        
    \multirow{3}{*}{3} 
    & Algorithm \ref{alg_mgsolver} for $A_{0}^{-1}$ 
    & 172	& 201	& 226	& 235 \\ 
    \cline{2-6}
    & Algorithm \ref{alg_mgsolverII} for $A_{0}^{-1}$
    & 172	& 198	& 221	& 224 \\ 
    \cline{2-6}
    & Identity
    & 118 & 1677 & $>3000$ & $>3000$ \\ 
    \hline\hline
  \end{tabular}
  }
  \caption{The iteration counts for PCG/CG methods in Example 3.}
  \label{tab_ex3steps}
\end{table}

\paragraph{\textbf{Example 4.}}
In this example, another three-dimensional elliptic interface problem
is considered, where the interface is given by the following level set
function (see Fig.~\ref{fig_3dinterface}), 
\begin{displaymath}
  \phi(x, y, z) = \left( (2.5(x - 0.5))^2 + (4.2(y - 0.5))^2 + (2.5(z
  - 0.5))^2 + 0.9\right)^2 - 64(y - 0.5)^2 - 1.3.
\end{displaymath}
The analytic solution and the coefficient are taken to be the same as
in Example 3. 
The numerical results under the energy norm and the $L^2$ norm are
depicted in Fig.~\ref{fig_ex4err}. 
The numerically detected convergence rates under both norms are
optimal, which match the theoretical analysis. 
For this test, the condition numbers of the resulting linear systems
are shown in Tab.~\ref{tab_cnexample4}. 
It can be observed that the condition number $\kappa(A_m)$ grows at the
speed $O(h^{-2})$, while the condition number $\kappa(A_0^{-1} A_m)$
increases only marginally as the mesh size $h$ tends to zero.
In Tab.~\ref{tab_ex4steps}, we list the iteration counts for PCG/CG
solvers. 
In accordance with the condition number, the PCG
iteration counts for convergence are numerically observed to be independent of the mesh
size.

\begin{figure}[htbp]
  \centering
  \includegraphics[width=0.32\textwidth]{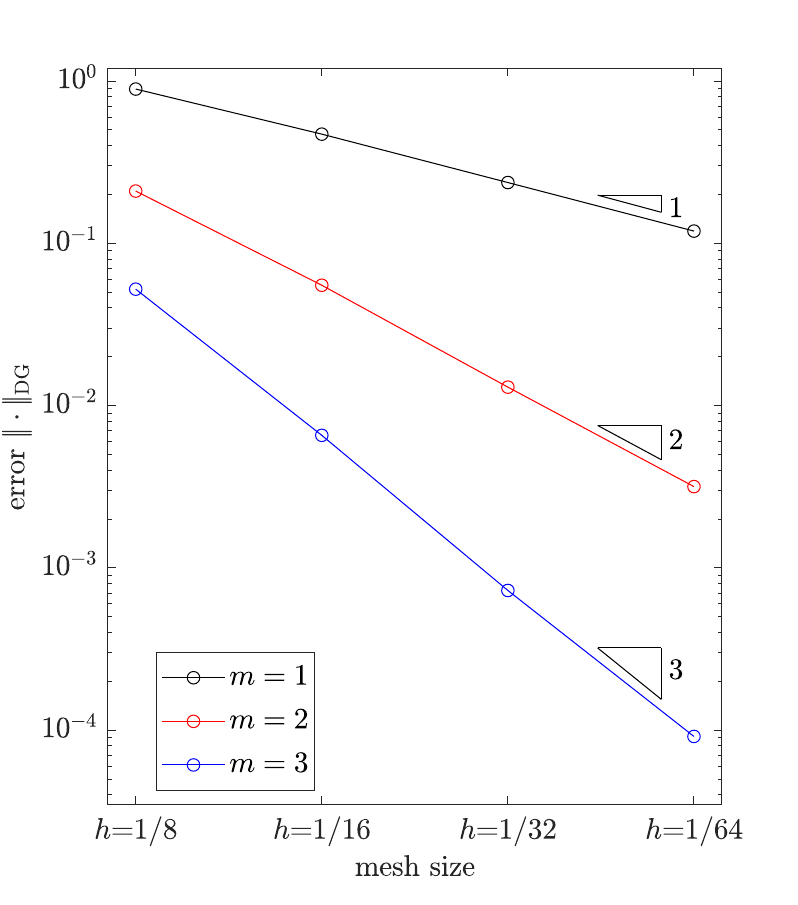}
  \hspace{50pt}
  \includegraphics[width=0.32\textwidth]{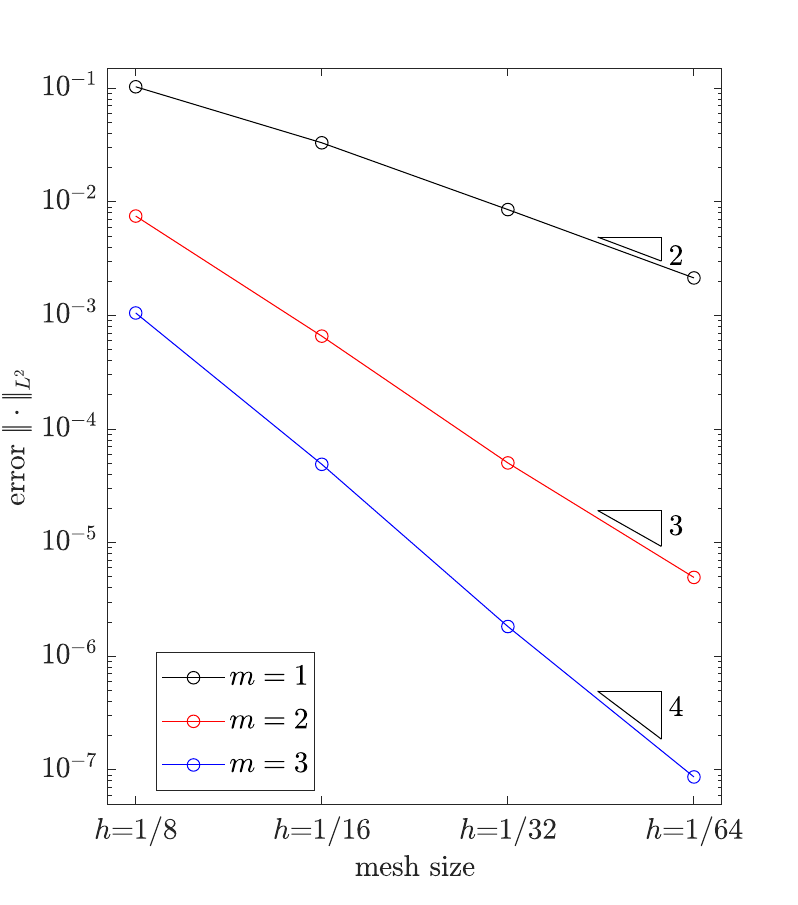}
  \caption{The numerical errors under the energy norm (left)/$L^2$ norm
  (right) in Example 4.}
  \label{fig_ex4err}
\end{figure}

\begin{table}[htb]
  \centering
  \renewcommand{\arraystretch}{1.28}
  \scalebox{0.78}{
  \begin{tabular}{p{0.5cm}|p{2.0cm}|p{1.8cm}|p{1.8cm}|p{1.8cm}| 
    p{1.8cm}}
        \hline\hline
    \multicolumn{2}{c|}{\diagbox[width=3.2cm, height=0.7cm]{$m$}{$h$}} 
    & 1/8 & 1/16 & 1/32 & 1/64 \\
    \hline
    \multirow{2}{*}{1} & $\kappa(A_{0}^{-1} A_m)$ & 
    19.68 & 25.20 & 31.65 & 36.98 \\ 
    \cline{2-6}
    &$\kappa(A_m)$ & 6.241e+2 & 2.259e+3 & 9.109e+3 & 3.583e+4 \\ 
    \hline
    \multirow{2}{*}{2} & $\kappa(A_{0}^{-1} A_m)$ &
    76.49 & 103.64 & 108.37 & 143.35 \\ 
    \cline{2-6}
    &$\kappa(A_m)$ 
    & 1.205e+3 & 4.650e+3 & 1.504e+4 & 5.910e+4 \\ 
    \hline
    \multirow{2}{*}{3} & $\kappa(A_{0}^{-1} A_m)$ 
    & 385.67 & 625.88 & 646.12 & 720.47 \\ 
    \cline{2-6}
    &$\kappa(A_m)$ & 3.025e+3 & 1.114e+4 & 3.416e+4 & 1.320e+5 \\ 
    \hline \hline
    \end{tabular}
   }
  \caption{The condition numbers of preconditioned/nonpreconditioned
  linear systems in Example 4.}
  \label{tab_cnexample4}
\end{table}

\begin{table}[htb]
  \centering
  \renewcommand{\arraystretch}{1.28}
  \scalebox{0.78}{
  \begin{tabular}{p{0.25cm}|p{3.5cm}|p{1.35cm}|p{1.35cm}|p{1.35cm}|
    p{1.35cm}}
    \hline\hline
    \multirow{2}{*}{$m$} & \diagbox[width=3.9cm, height=0.9cm]{Preconditioner}{$h$}
    & 1/8 & 1/16 & 1/32 & 1/64 \\ \hline
    \multirow{3}{*}{1} 
    & Algorithm \ref{alg_mgsolver} for $A_{0}^{-1}$ 
    & 47	& 55	& 65	& 72 \\ 
    \cline{2-6}
    & Algorithm \ref{alg_mgsolverII}  for $A_{0}^{-1}$ 
    & 47	& 54	& 61	& 64 \\ 
    \cline{2-6}
    & Identity  & 128 & 622 & 1258 & 2405\\ \hline        
    \multirow{3}{*}{2} 
    & Algorithm \ref{alg_mgsolver} for $A_{0}^{-1}$ 
    & 85	& 107	& 115	& 119 \\ 
    \cline{2-6}
    & Algorithm \ref{alg_mgsolverII} for $A_{0}^{-1}$
    & 85	& 107	& 114	& 113 \\ 
    \cline{2-6}
    & Identity 
    & 101 & 900 & 1641 & $>3000$ \\ \hline        
    \multirow{3}{*}{3} 
    & Algorithm \ref{alg_mgsolver} for $A_{0}^{-1}$ 
    & 162	& 196	& 211	& 231 \\ 
    \cline{2-6}
    & Algorithm \ref{alg_mgsolverII} for $A_{0}^{-1}$
    & 162	& 195	& 208	& 218 \\ 
    \cline{2-6}
    & Identity
    & 110 & 1449 & $>3000$ & $>3000$ \\ 
    \hline\hline
  \end{tabular}
  }
  \caption{The iteration counts for PCG/CG methods in Example 4.}
  \label{tab_ex4steps}
\end{table}

\paragraph{\textbf{Efficiency Comparison}.}
The main feature of our proposed method is that each element has only
one degree of freedom, independent of the degree $m$.
Hughes et al. \cite{hughes2000comparison} point out that the number of
unknowns for a discretized problem is a reasonable indicator for
evaluating the computational efficiency of a numerical method.
We make a numerical comparison between the proposed reconstructed
method and the unfitted discontinuous Galerkin method
\cite{Gurkan2019stabilized}. 
Both methods are tested on Example 1 and Example 3, as the
two-dimensional and three-dimensional test cases, respectively. 
In Fig.~\ref{fig_compare}, we plot the $L^2$ numerical errors 
for both methods against the number of degrees of freedom with 
different degrees. 
It can be seen that in two and three dimensions, the reconstructed
method always has better efficiency, in the sense that 
it uses fewer degrees of freedom to achieve a comparable
numerical error. 
In Tab.~\ref{tab_compare}, we list the ratio of the number of degrees
of freedom required by the two methods to reach the same numerical
error. 
The advantage in the approximation efficiency becomes more remarkable
for higher-order accuracy.

\begin{figure}[htb]
  \centering
  \includegraphics[width=.32\textwidth]{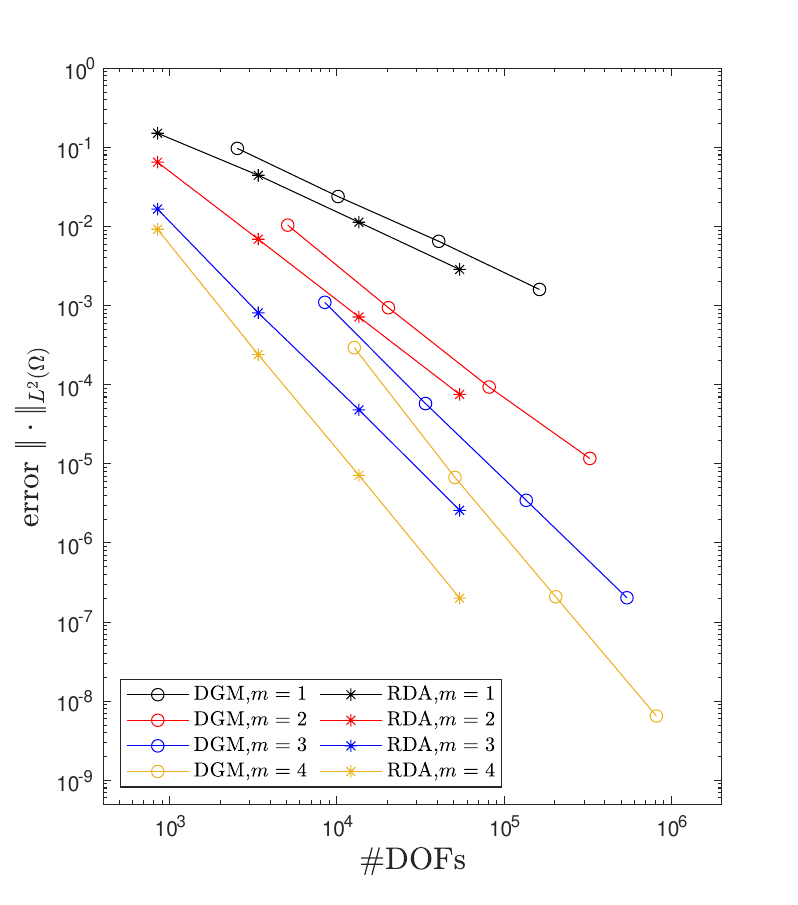}
  \hspace{50pt}
  \includegraphics[width=.32\textwidth]{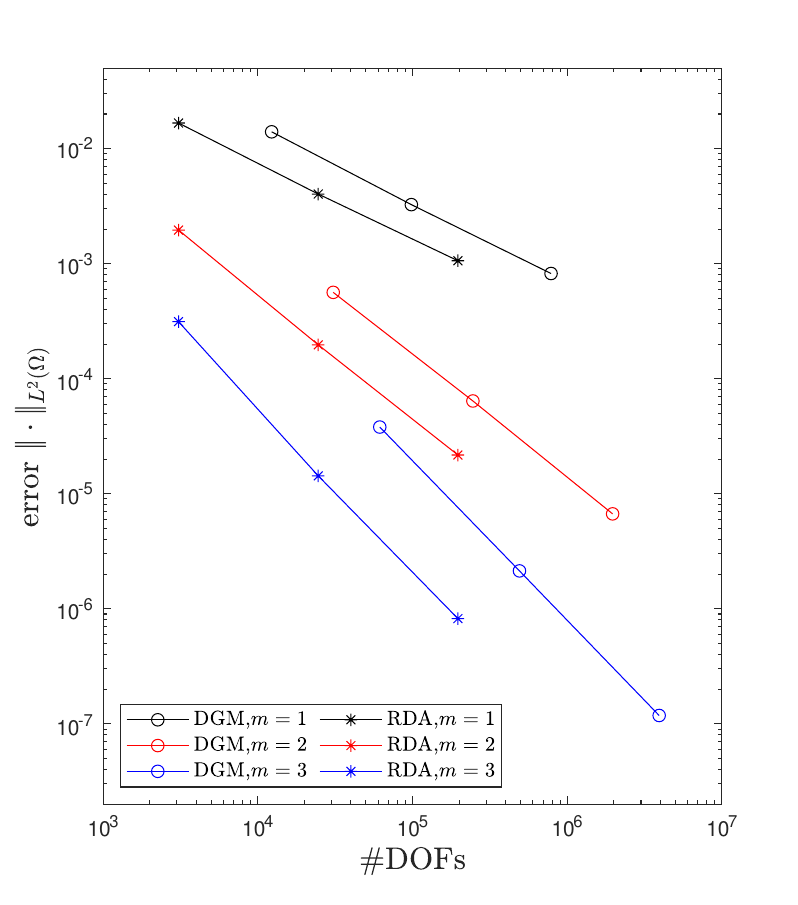}
  \caption{The $L^2$ numerical errors versus number of degrees of
  freedom for RDA/DG methods in two and three dimensions.}
  \label{fig_compare}
\end{figure}

\begin{table}[htp]
\begin{minipage}[t]{0.3\textwidth}
  \centering
  \renewcommand{\arraystretch}{1.25}
  \scalebox{0.8}{
  \begin{tabular}{p{1.8cm}|p{0.8cm}|p{0.8cm}|p{0.8cm}|p{0.8cm}}
    \hline\hline
    $m$ & 1 & 2 & 3 & 4 \\ 
    \hline
    RDA/DG & 84\% & 43\% & 37\% & 31\% \\ 
    \hline\hline
  \end{tabular}}
\end{minipage}
\hspace{2cm}
\begin{minipage}[t]{0.3\textwidth}
  \centering
  \renewcommand{\arraystretch}{1.25}
  \scalebox{0.8}{
  \begin{tabular}{p{1.8cm}|p{0.8cm}|p{0.8cm}|p{0.8cm}}
    \hline\hline
    $m$ & 1 & 2 & 3 \\
    \hline
    RDA/DG & 61\% & 36\% & 23\% \\
    \hline\hline
  \end{tabular}}
\end{minipage}
\caption{The ratio of the number of degrees of freedom involved in
  RDA/DG methods when achieving a comparable $L^2$ error in
  two and three dimensions.}
\label{tab_compare}
\end{table}


\section{Conclusions}
\label{sec_conclusion}
In this paper, we develop a preconditioned unfitted finite element
method for the elliptic interface problem, based on the reconstructed
discontinuous approximation. 
The approximation space is constructed by a patch reconstruction
process, where each element possesses only one degree of freedom. 
Thanks to the patch reconstruction, the stability near the interface
is naturally ensured, and we prove the optimal convergence rates.
Moreover, we construct a preconditioner from the piecewise constant
space for the high-order space, and the preconditioned linear system
is proven to be optimal and independent of how the interface cuts the
mesh.
For the lowest-order system, we propose a multigrid algorithm, which
can be used as a preconditioner for the high-order system. 
Numerical experiments in two and three dimensions illustrate the
convergence rates of the discretization and the efficiency of the
preconditioning method.



\begin{appendix}
  \section{}
  \label{sec_app}
  Here, we present more details and numerical tests for the constants
  $\Lambda_{m, K, i}$ and $\Lambda_m$. 
  We first briefly introduce their computing methods for a given
  element patch $\mS_K^i$. 
  For the element $K \in \MThi$, let $p_1, p_2, \ldots, p_l$ be a
  set of orthonormal basis functions in $\mP_m(K)$ under the
  inner product $(\cdot, \cdot)_{L^2(K)}$. 
  For any polynomial $q \in \mP_m(K)$, it can be expressed by
  coefficients $\bm{\alpha} = \{a_j\}_{j = 1}^l \in \mb{R}^l$ that $q
  = \sum_{j =1}^l a_j p_j$.
  The polynomials $q$ and $p_j$ can be directly extended to the domain
  $\mD_K^i$. 
  From the definition \eqref{eq_Lambda}, the constant $\Lambda_{m, K,
  i}$ can be equivalently written into an algebraic form, 
  \begin{displaymath}
    \Lambda_{m, K, i}^2 = \max_{\bm{\alpha} \in \mb{R}^l}
    \frac{ | \bm{\alpha} |_{l^2}^2 }{ h_K^d \bm{\alpha}^T B_{K, i}
    \bm{\alpha}}, 
  \end{displaymath}
  where the matrix $B_{K, i}$ is given as 
  \begin{displaymath}
    B_{K, i} = \{b_{jk}^i \}, \quad b_{jk}^i = \sum_{K' \in \mS_K^i}
    p_j(\bm{x}_{K'})p_k(\bm{x}_{K'}).
  \end{displaymath}
  Then, there holds $\Lambda_{m, K, i} = (h_K^d \sigma_{\min}(B_{K,
  i}))^{-1/2}$, where $\sigma_{\min}(\cdot)$ denotes the smallest
  singular value for a given matrix.
  Consequently, the constant $\Lambda_m$ can be readily obtained by
  computing the smallest singular value for all $B_{K, i}$.

  In Figure~\ref{fig_lambda2d}, we plot the stability constant $\Lambda_m$
  against the given threshold $\mN_m$ on the mesh with $h = 1/40$ for
  degrees $1 \leq m \leq 4$. Here, $\mN_m$ is a prespecified parameter
  that controls the size of the element patch. 
  It is clear that $\Lambda_m$ is nearly constant when the threshold
  $\mN_m$ is large enough, which verifies the theoretical prediction
  given in Section \ref{sec_space}.
  From the numerical observation, $\mN_m$ is close to $\dim(\mb{P}_m)
  + m + 1$ in two dimensions.
  In Figure~\ref{fig_lambda3d}, we depict the constant $\Lambda_m$
  against the threshold $\mN_m$ in three dimensions with the mesh size
  $h = 1/16$. 
  Similarly, $\Lambda_m$ is bounded and increases very slightly if
  $\mN_m$ is larger than a specific value.
  In three dimensions, the threshold $\mN_m$ is approximately
  $1.5\dim(\mb{P}_m)$. 
  In the computer implementation, we can use the condition 
  $\max_{K \in \MThi} \Lambda_{m, K, i} \leq 5\min_{K \in \MThi}
  \Lambda_{m, K, i}$ as the criterion to check whether the element
  patches are appropriate.
  If this condition is violated, we increase the threshold until it is
  satisfied.

  \begin{figure}[htbp]
    \centering
    \includegraphics[width=0.23\textwidth]{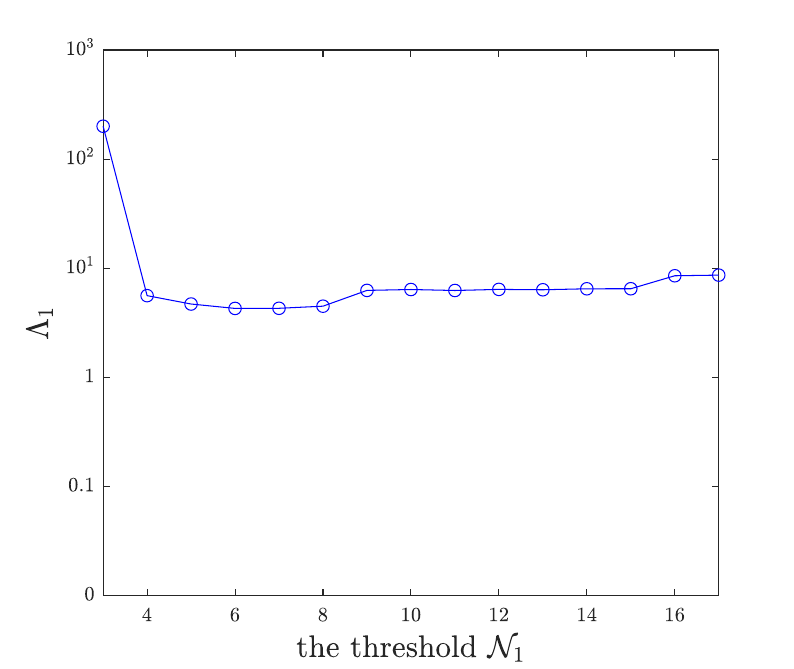}
    \hfill
    \includegraphics[width=0.23\textwidth]{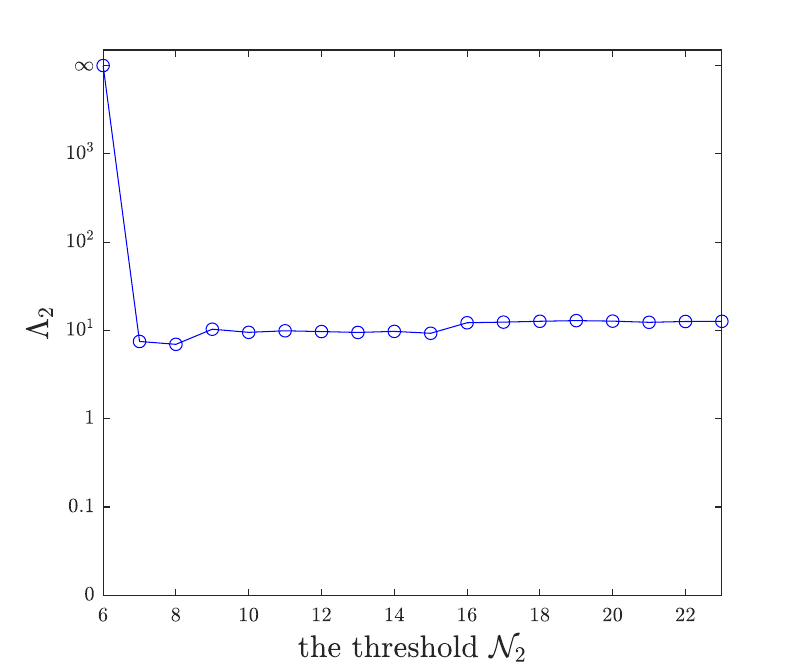}
    \hfill
    \includegraphics[width=0.23\textwidth]{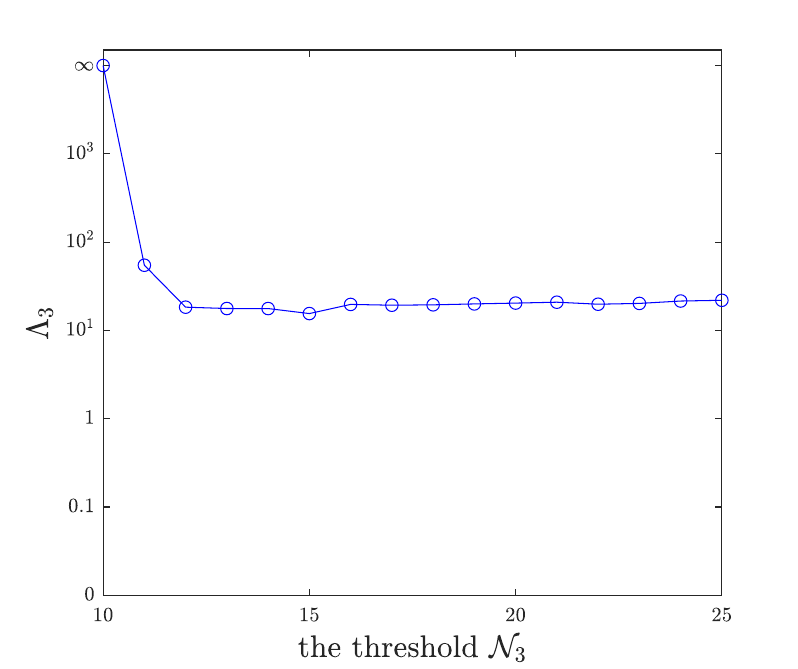}
    \hfill
    \includegraphics[width=0.23\textwidth]{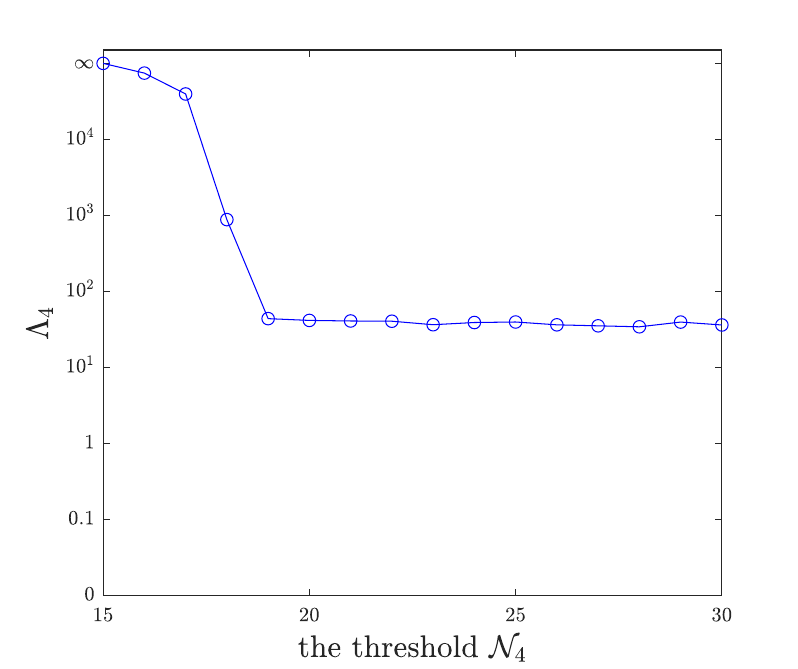}
    \caption{The stability constant $\Lambda_m$ in two dimensions.}
    \label{fig_lambda2d}
  \end{figure} 

  \begin{figure}[htbp]
    \centering
    \includegraphics[width=0.23\textwidth]{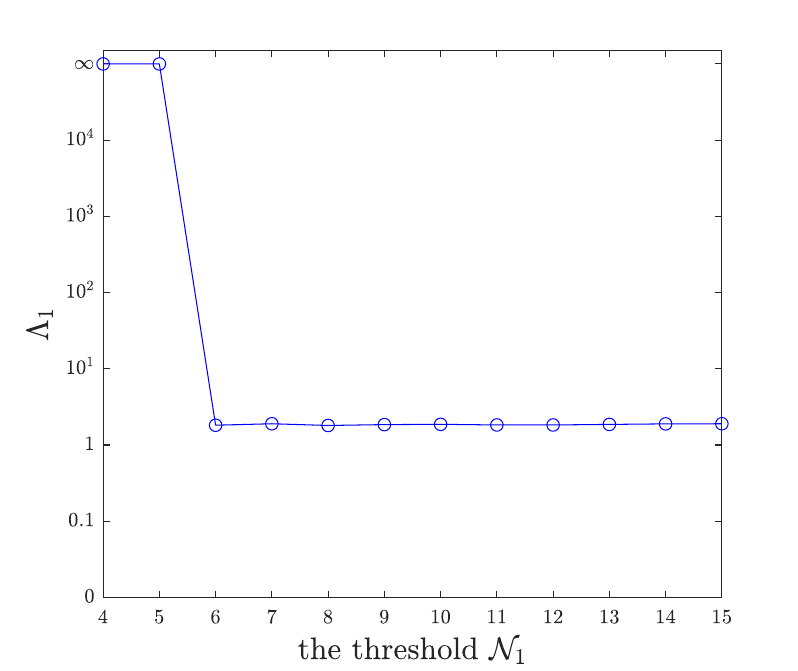}
    \hspace{10pt}
    \includegraphics[width=0.23\textwidth]{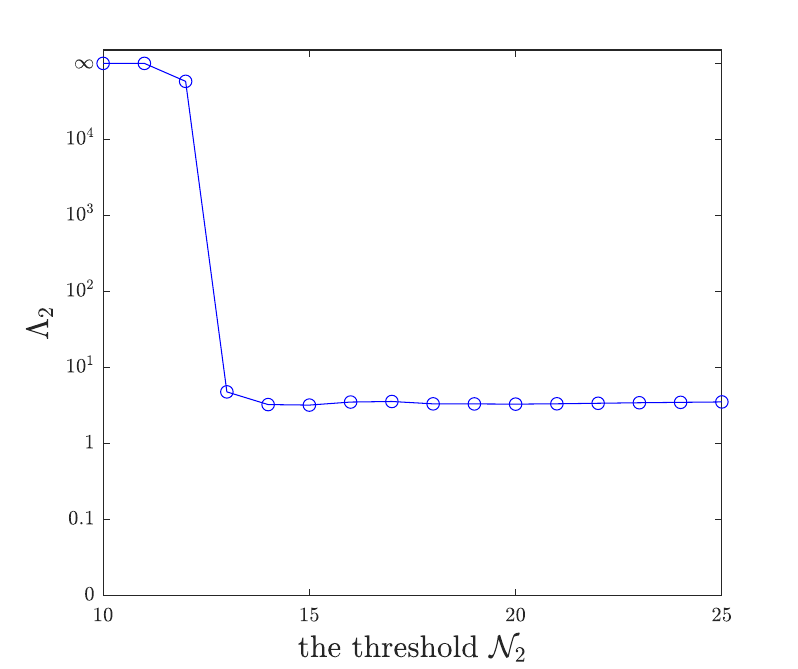}
    \hspace{10pt}
    \includegraphics[width=0.23\textwidth]{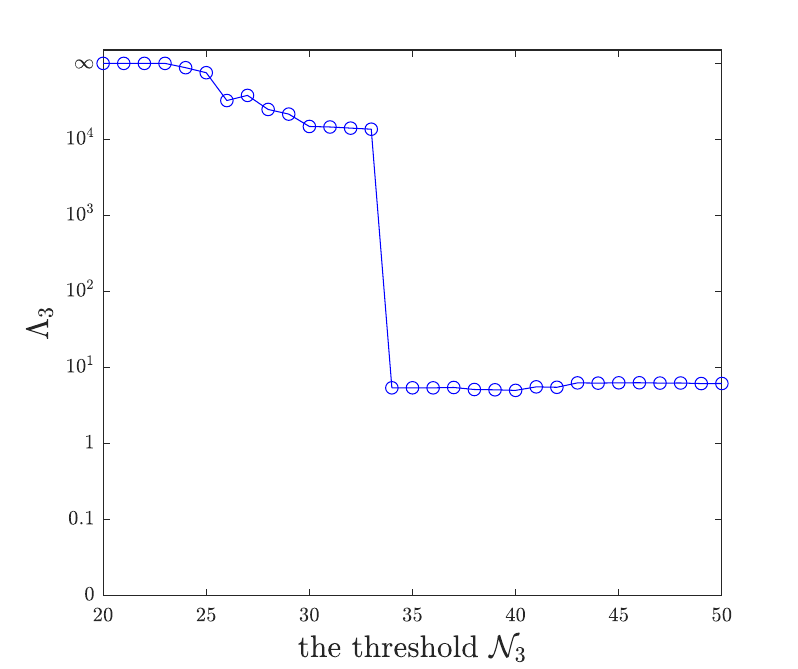}
    \caption{The stability constant $\Lambda_m$ in three dimensions.}
    \label{fig_lambda3d}
  \end{figure}

\end{appendix}

\bibliographystyle{amsplain}
\bibliography{../ref}

\end{document}